\documentclass{article}
\usepackage[utf8]{inputenc}
\usepackage{fullpage}
\usepackage{amssymb,amsmath,amsthm}
\usepackage{amsfonts}
\usepackage{latexsym}
\usepackage{amsfonts}
\usepackage{xcolor}
\usepackage{graphicx}
\usepackage{verbatim}
\usepackage{enumitem}
\usepackage{bbm}
\usepackage{hyperref}

\newtheorem{theorem}{Theorem}[section] 

\newtheorem{lemma}[theorem]{Lemma}
\newtheorem{corollary}[theorem]{Corollary}

\newtheorem{claim}[theorem]{Claim}

\newtheorem{remark}[theorem]{Remark}

\newtheorem{construction}[theorem]{Construction}

\title{Saturation for the $3$-Uniform Loose $3$-Cycle}

\author{Sean English\footnote{University of Illinois at Urbana--Champaign, Urbana, IL. E-mail: {\tt senglish@illinois.edu}.}, Alexandr Kostochka\thanks{
\footnotesize {University of Illinois at Urbana--Champaign, Urbana, IL
 and Sobolev Institute of Mathematics, Novosibirsk 630090, Russia. E-mail: \texttt {kostochk@math.uiuc.edu}.
 Research
is supported in part by  NSF RTG Grant DMS-1937241.
}},
Dara Zirlin\footnote{University of Illinois at Urbana--Champaign, Urbana, IL.}}

\makeatletter
\newcommand{\subjclass}[2][2020]{%
  \let\@oldtitle\@title%
  \gdef\@title{\@oldtitle\footnotetext{#1 \emph{Mathematics subject classification.} #2}}%
}
\newcommand{\keywords}[1]{%
  \let\@@oldtitle\@title%
  \gdef\@title{\@@oldtitle\footnotetext{\emph{Key words.} #1.}}%
}
\makeatother

\subjclass{05C35, 05D99}

\keywords{Saturation, Loose Cycles, Hypergraphs, Discharging}

\date{}

\begin{document}

	\maketitle
	
\vspace{-.9cm}	
	\begin{flushright}
{\em	Dedicated to the memory of Landon Rabern}
	\end{flushright}
\vspace{0.7cm}

	\begin{abstract}
		Let $F$ and $H$ be $k$-uniform hypergraphs. We say $H$ is $F$-saturated if $H$ does not contain a subgraph isomorphic to $F$, but $H+e$ does for any hyperedge $e\not\in E(H)$. The saturation number of $F$, denoted $\mathrm{sat}_k(n,F)$, is the minimum number of edges in a $F$-saturated $k$-uniform hypergraph $H$ on $n$ vertices. Let $C_3^{(3)}$ denote the $3$-uniform loose cycle on $3$ edges. In this work, we prove that
		\[
		\left(\frac{4}3+o(1)\right)n\leq \mathrm{sat}_3(n,C_3^{(3)})\leq \frac{3}2n+O(1).
		\]
		This is the first non-trivial result on the saturation number for a fixed short hypergraph cycle.
	\end{abstract}
	
	\section{Introduction}
	
	Let $F$ and $H$ be $k$-uniform hypergraphs. We say $H$ is \emph{$F$-free} if $H$ does not contain $F$ as a sub-hypergraph. One of the central problems in extremal combinatorics is to determine the \emph{Tur\'an number} of $H$, denoted $\mathrm{ex}_k(n,F)$, and defined
	\[
	\mathrm{ex}_k(n,F)=\max \{|E(H)|:H\text{ is a }F\text{-free hypergraph on }n\text{ vertices}\}.
	\]
	We say that $H$ is $F$-saturated if $H$ is $F$-free, but $H+e$ contains a copy of $F$ for every hyperedge $e\not\in E(H)$. Since each $F$-free graph $H$ with $E(H)=\mathrm{ex}_k(|V(H)|,F)$ is $F$-saturated,  Tur\'an numbers can be defined in terms of maximizing the number of edges over $F$-saturated graphs rather than $F$-free graphs. A consequence of this phrasing of the definition is that it leads to a natural minimization problem related to Tur\'an numbers. Originally introduced by Erd\H os, Hajnal and Moon~\cite{EHM1964} using different terminology, {\em the saturation number}, $\mathrm{sat}_k(n,F)$ is defined by
	\[
	\mathrm{sat}_k(n,F)=\min\{|E(H)|:H\text{ is a }F\text{-saturated hypergraph on }n\text{ vertices}\}.
	\]
	K\' aszonyi and  Tuza~\cite{KT1986} proved that
	 $\mathrm{sat}_2(n,F)=O(n)$, and then Pikhurko~\cite{P1999} proved that for general $k$, $\mathrm{sat}_k(n,F)=O(n^{k-1})$.
	
	In the seminal paper~\cite{EHM1964}, 
	Erd\H os, Hajnal and Moon determined the saturation numbers for graph cliques  exactly. The first result for saturation of $k$-uniform hypergraphs is due to Bollob\'as~\cite{B1965}, and in this work, the method known as the \emph{set-pair method}~\cite{T1994} was first developed. In addition to complete graphs, different trees have received careful study~(e.g. \cite{FFGJ2009, KT1986}). In terms of hypergraphs, most of the specific saturation numbers determined outside of complete graphs involve forbidden families of hypergraphs, such as triangular families~\cite{P2004}, intersecting hypergraphs~\cite{DDFL1985}, and very recently Berge hypergraphs (e.g. \cite{AE2019, AW2019, EGMT2019, EGGMS2019, GPTV2022}). For a detailed dynamic survey on all aspects of saturation in graphs and hypergraphs, see~\cite{FFS2011}.

\subsection{Saturation for Cycles}

One of the families of graphs that have received the most attention in saturation literature is cycles. While cycles have received considerable attention, very few exact results have been obtained. 
For specific short graph cycles, the following bounds are known (We assume $n$ is large enough for  results below):
\begin{itemize}
    \item $\mathrm{sat}_2(n,C_3)=\mathrm{sat}_2(K_3,n)=n-1$,~Erd\H os,  Hajnal, and  Moon~\cite{EHM1964},
    \item $\mathrm{sat}_2(n,C_4)=\left\lfloor\frac{3n-5}{2}\right\rfloor$,~Ollmann~\cite{O1972},
    \item $\mathrm{sat}_2(n,C_5)=\left\lceil\frac{10}{7}(n-1)\right\rceil$,~Chen~\cite{C2009},
    \item $\left\lceil\frac{7n}6\right\rceil-2\leq \mathrm{sat}_2(n,C_6)\leq\left\lceil\frac{10}{7}(n-1)\right\rceil$,~
    Gould, \L uczak,   Schmitt~\cite{GLS2006} and Zhang,  Luo,   Shigeno~\cite{ZLS2015}.
\end{itemize}

Aside from these small cases, the best-known  bounds on $\mathrm{sat}_2(n,C_\ell)$ for fixed $\ell$ are 
obtained by F\" uredi  and   Kim~\cite{FK2013}:
\[
\left(1+\frac{1}{\ell+2}\right)n-1<\mathrm{sat}_2(n,C_\ell)<\left(1+\frac{1}{\ell-4}\right)n+\binom{\ell-2}{2}.
\]
Of note is the fact that even for graphs, the asymptotics for saturation numbers of cycles is not known already for cycles of length $6$ or more. Also of note, for the $5$-cycle, through a technical feat, it was shown that there are exactly $29$ distinct minimal constructions, some of which are specific graphs that only work for one value of $n$, others which constitute infinite families~\cite{C2011}. This highlights a difficulty in studying the saturation function in general and for cycles - one usually does not expect  to prove a nice stability result when there are multiple different extremal examples.

Saturation numbers for the family of all cycles of length above a certain value $\ell$ have also been studied, with exact results determined for  $3\leq \ell\leq 6$~\cite{FJMTW2012, MHHG2021}. In addition to these results involving cycles of short length, many results on the saturation numbers of Hamiltonian cycles have been studied, with numerous results leading up to proving that $\mathrm{sat}_2(n,C_n)=\left\lceil\frac{3n}{2}\right\rceil$ (upper bound given first in~\cite{B1972}, while the lower bound can be found in~\cite{LJZY1997}).

When passing from graphs to hypergraphs, there are many ways to generalize the notion of a cycle. One of the more general notions of a cycle in a hypergraph is an $r$-overlapping cycle. Namely, the {\em $k$-uniform $r$-overlapping cycle on $\ell$ edges} is the unique $k$-uniform hypergraph on $\ell(k-r)$ vertices and $\ell$ edges such that there exists an ordering of the vertex set, say $v_1,v_2,\dots,v_{\ell(k-r)}$ such that $e_i=\{v_{(k-r)(i-1)+1},v_{(k-r)(i-1)+2},\dots,v_{(k-r)(i-1)+k}\}$ is an edge for each $1\leq i\leq \ell$ (indices taken modulo $\ell(k-r)$). When $r=1$, we will simply call this hypergraph the $k$-uniform \textbf{loose} cycle on $\ell$ edges, and denote it by $C_\ell^{(k)}$.

Up until this work, the entire literature involving saturation for $r$-overlapping cycles has been for Hamiltonian cycles (for a $k$-uniform $r$-overlapping Hamiltonian cycle to exist in an $n$-vertex graph, we must have $(k-r)\mid n$). In this setting, due to the difficulty in such problems, the work has been mostly focused on determining the order of magnitude for which these saturation numbers grow (see e.g.~\cite{DZ2012, KK1999, RZ2016} for some of the results in this direction).

\subsection{Main Result}

In this work, we study the saturation function for a short loose cycle, namely $C_3^{(3)}$. Our main result is as follows.
	\begin{theorem}\label{main}
		We have that
		\[
		\left(\frac{4}3+o(1)\right)n\leq \mathrm{sat}_3(n,C_3^{(3)})\leq \frac{3}2n+O(1).
		\]
	\end{theorem}

To the authors best knowledge, this is the first non-trivial result on saturation numbers for a specific hypergraph cycle of fixed length (the first author and others did provide some bounds on saturation for short Berge hypergraphs cycles in~\cite{EGGMS2019}, but in general results involving families of hypergraphs tend to be easier than results involving a single hypergraph).

It is worth noting that at least for small values of $n$, the $o(1)$ in the lower bound is necessary, as for example one might note that $\mathrm{sat}_3(9,C_3^{(3)})=6$ (See Figure~\ref{figure small example} for the optimal construction). 
The authors think  that for large $n$, 
the upper bound is more likely to be asymptotically correct than the lower bound.

\begin{figure}
\begin{center}
\includegraphics[width=4cm]{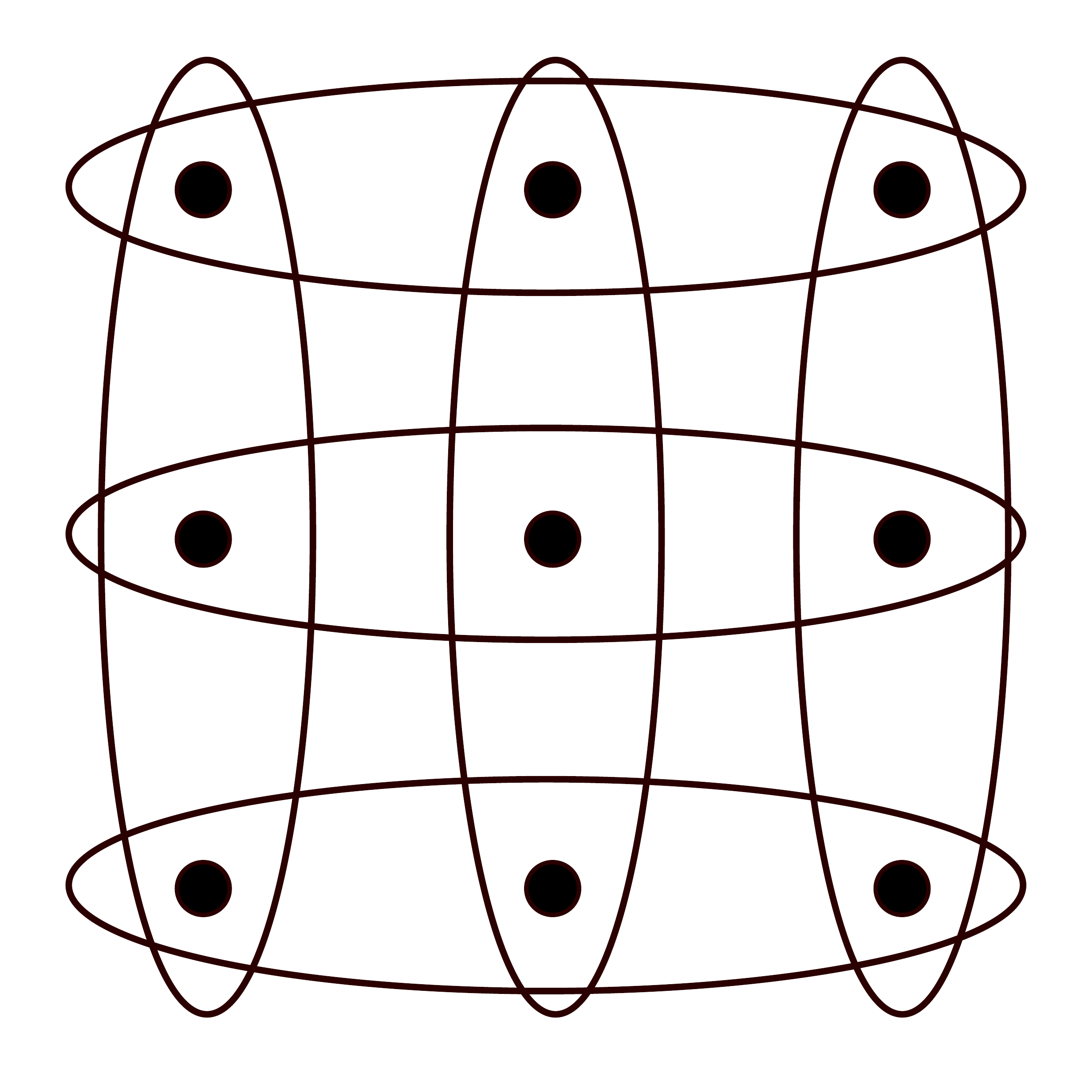}
\caption{A $C_3^{(3)}$-saturated hypergraph with very few edges exhibiting that $\mathrm{sat}_3(9,C_3^{(3)})\leq 6$.}\label{figure small example}
\end{center}
\end{figure}

The proof of our main result is broken up as follows. In Section~\ref{section upper bound}, we provide a relatively straightforward deterministic construction on $\frac{3}{2}n+O(1)$ edges and show that this construction is saturated. In the remaining sections of the paper we show the  more involved lower bound. The proof  uses the technique of \textbf{asymptotic discharging}. For a primer on discharging, see~\cite{CW2017}. Normally in a discharging proof, it would be shown that certain small structures are \emph{reducible}, i.e. they cannot exist in a minimum counterexample to a proposition. Here we allow many configurations to exist, but only in numbers small enough that their existence does not affect the leading term of the final bound. In Section~\ref{section lower bound proof sketch}, we provide a high-level proof sketch that goes through the main ideas of the proof without the technical details. Then in Sections~\ref{section preliminary lemmas}, \ref{section 3 vertex} and \ref{section 2 vertex}, we provide the main structural results necessary for the discharging proof, and finally in Section~\ref{section discharging}, we give our discharging scheme and provide the proof of the lower bound using this discharging scheme.

\subsection{Notation and Definitions}

Given a cycle $C_3^{(3)}$, we will call the vertices of degree $2$ the {\em core vertices} of the cycle. We may refer to $C_3^{(3)}$ as a \emph{triangle}. 

Vertices of degree $i$ in a graph $G$ will be called {\em $i$-vertices}, and $i$-vertices adjacent to a vertex $v$ will be called
{\em $i$-neighbors} of $v$. For $A\subseteq V(G)$, $N(A)$ denotes the set of vertices  $u\in V(G)\setminus A$ such that some edge of $G$ contains $u$ and some vertex in $A$, and $N[A]=A\cup N(A)$. Given a pair of vertices $u$ and $v$, we write $d(uv)$ to denote the co-degree of the pair, i.e. the number of edges that contain both $u$ and $v$. We may say that $u$ is a double neighbor  or triple neighbor  of $v$ if $d(uv)\geq 2$ or $d(uv)\geq 3$ respectively.

Given a graph $G$ and sets $A,B,C\subseteq V(G)$, we will say an edge $e=\{a,b,c\}\in E(G)$ is an $(A,B,C)$ edge if (after possibly renaming) $a\in A$, $b\in B$ and $c\in C$. If one of the sets $A$, $B$ or $C$ is of the form $\{v\in V(G)\mid d(v)=d\}$ for some $d\in \mathbb{N}$, we will often just write the number $d$ in place of the set. For example, we may say an edge $e$ is an $(A,4,2)$ edge if $e$ contains one vertex in the set $A$, one vertex of degree $4$, and one vertex of degree $2$. Finally, if one of the sets is of the form $\{v\}$ for some $v\in V(G)$, we will simply write $v$ in place of the set.

For $u,v\in V(G)$, a {\em $u,v$-link} is a $2$-edge loose path $L$ from $u$ to $v$. The common vertex of the two edges of $L$ is the {\em center of $L$}. If there exists a $u,v$-link in $G$, we will say $uv$ is a \emph{good pair}, and if not, we will say $uv$ is a \emph{bad pair}.

\section{Upper Bound - A Construction}\label{section upper bound}

Let $n\geq 14$ be an integer. We now present a construction, $G_n$, that has $n$ vertices and $\frac{3}{2}n+O(1)$ edges.

\begin{construction}\label{construction saturated hypergraph}
Let $m$ and $c$ be integers such that $2\leq c\leq 5$ and $n=4m+c$. For each $i$ with $1\leq i\leq m+2-c$, let $A_i$ denote the $3$-uniform hypergraph on $6$ vertices and $6$ edges with $V(A_i)=\{x, y, a_{x,i}, a_{y,i}, a_{1,i}, a_{2,i}\}$ and
\[
E(A_i)= \big\{\{x,a_{x,i},a_{y,i}\},
\{y,a_{x,i},a_{y,i}\},\{x,a_{x,i},a_{1,i}\},\{x,a_{x,i},a_{2,i}\},\{y,a_{y,i},a_{1,i}\},\{y,a_{y,i},a_{2,i}\}\big\}.
\]
Furthermore, for each $i$ with $1\leq i\leq c-2$, let $B_i$ denote the hypergraph on $7$ vertices and $9$ edges such that $V(B_i)=\{x,y,b_{x,i},b_{y,i},b_{1,i},b_{2,i},b_{3,i}$ and
\[
E(B_i)=\big\{\{x,b_{x,i},b_{y,i}\},\{y,b_{x,i},b_{y,i}\},\{b_{1,i},b_{2,i},b_{3,i}\}\big\}\cup\bigcup_{j=1}^3\big\{\{x,b_{x,i},b_{j,i}\},\{y,b_{y,i},b_{j,i}\}\big\}.
\]
Now, let $G_n$ be union of the $A_i$'s and $B_i$'s, and note that 
\[
|V(G_n)|=2+4(m+2-c)+5(c-2)=4m+c=n,
\]
while
\[
|E(G_n)|=6(m-2+c)+9(c-2)=\begin{cases}
\frac{3}{2}n&\text{ if }n\equiv 0 \mod 4,\\
\frac{3}{2}n+\frac{3}{2}&\text{ if }n\equiv 1 \mod 4,\\
\frac{3}{2}n-3&\text{ if }n\equiv 2 \mod 4,\\
\frac{3}{2}n-\frac{3}{2}&\text{ if }n\equiv 3 \mod 4.\\
\end{cases}
\]
We call each of the subgraphs $A_i$ or $B_i$ a \textbf{brick} of $G_n$. See Figure~\ref{figure bricks} for drawings of the two types of bricks in $G_n$.
\end{construction}

\begin{figure}
\begin{center}
\includegraphics[width=4.1cm]{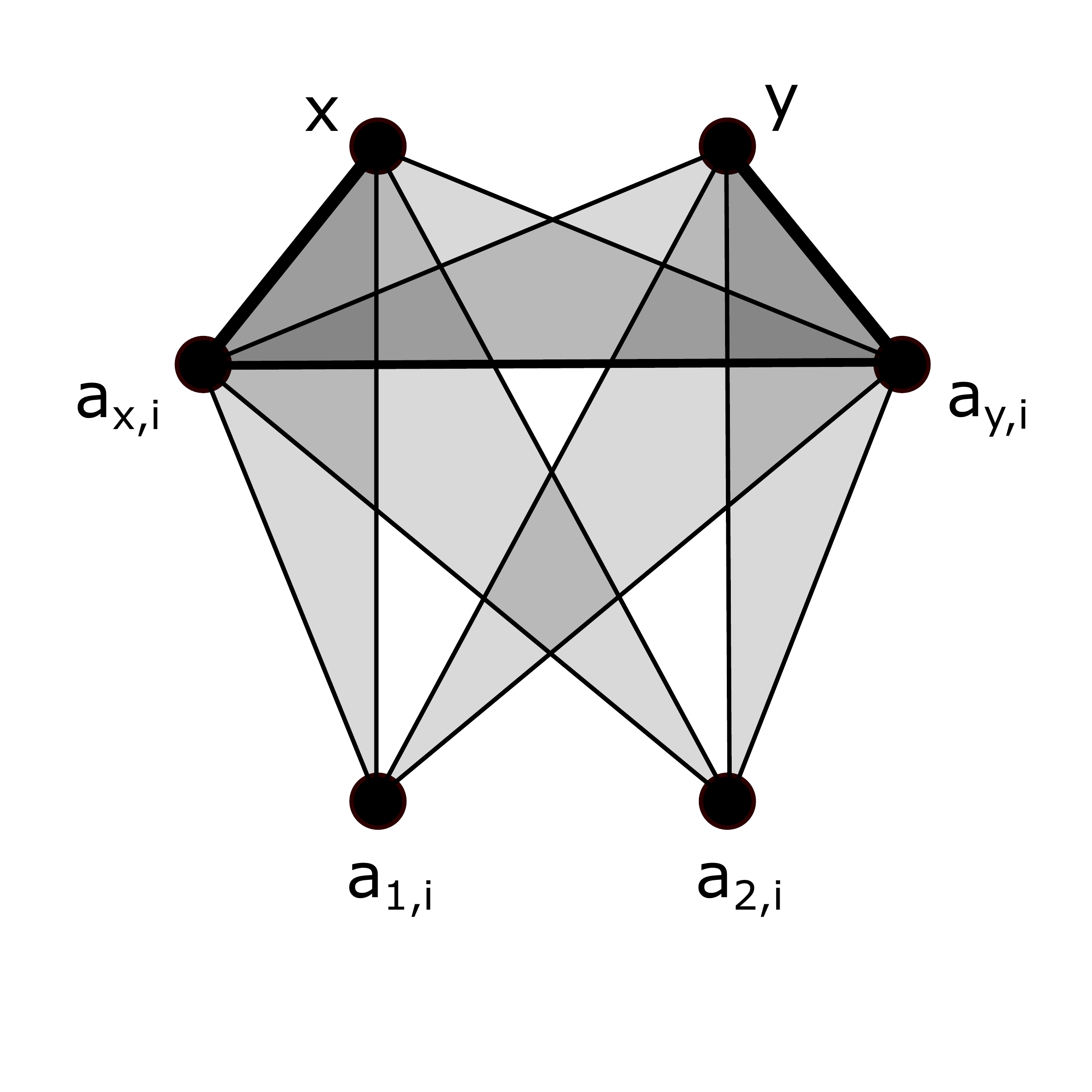}
\includegraphics[width=4cm]{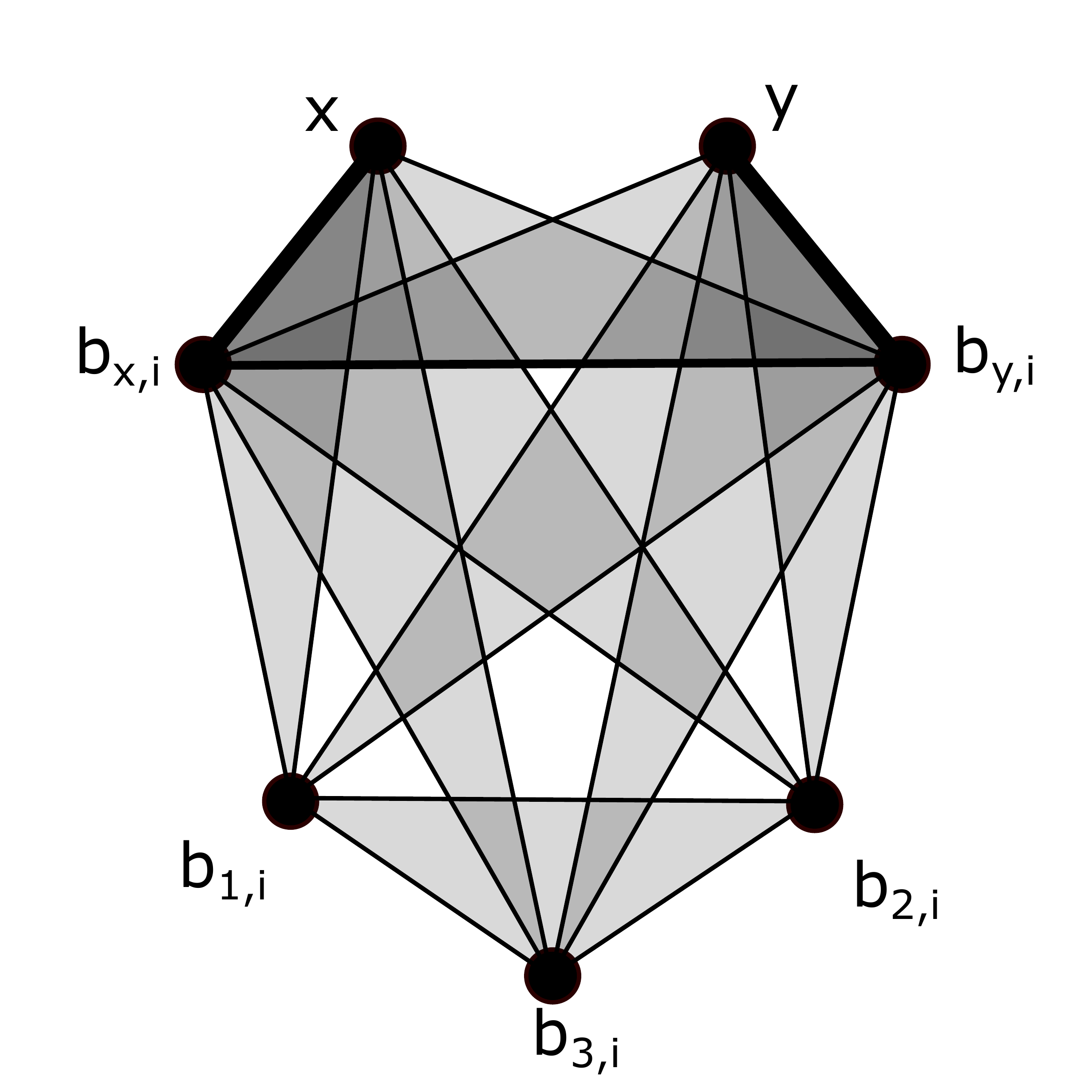}
\caption{The bricks $A_i$ and $B_i$ from Construction~\ref{construction saturated hypergraph}.}\label{figure bricks}
\end{center}
\end{figure}

\begin{theorem}\label{theorem upper bound}
The hypergraph $G_n$ from Construction~\ref{construction saturated hypergraph} is $C_3^{(3)}$-saturated for all $n\geq 14$. Consequentially,
\[
\mathrm{sat}_3(n,C_3^{(3)})\leq \frac{3}{2}n+O(1).
\]
\end{theorem}

\begin{proof}
First we will show that $G_n$ is $C_3^{(3)}$-free. Assume to the contrary that $G_n$ contains a copy, $T$, of $C_3^{(3)}$, and note that all the core vertices of $T$ must be contained in the same brick of $G_n$ since they are adjacent. Furthermore, $x$ and $y$ cannot both be core vertices, so this implies $T$ must be completely contained inside one brick of $G_n$. However it can be directly verified that neither type of brick contains a copy of $C_3^{(3)}$. Thus, $G_n$ is $C_3^{(3)}$-free.

Now we will show that $G_n$ is $C_3^{(3)}$-saturated. Let $e$ be any triple of vertices such that $e\not\in E(G_n)$. First note that every brick of $G_n$ contains an $x,y$-link, so if both $x$ and $y$ are in $e$, then this creates a $C_3^{(3)}$, so we may assume at least one of $x$ or $y$ is not in $e$. Furthermore, under this assumption, it can be directly verified that if all three vertices in $e$ are contained in a single brick of $G_n$, this creates a $C_3^{(3)}$ within that brick, so we may assume otherwise.

\textbf{Case 1:} $e$ contains exactly one of $x$ or $y$. Assume without loss of generality that $x\in e$. Let $u,v\in e$ be the other two vertices of $e$, and note that $u$ and $v$ must be in different bricks of $G_n$. However, every vertex in every brick is in an edge that contains $y$ but not $x$. Thus $e$, along with the edge containing $y$ and $u$, but not $x$, and the edge containing $y$ and $v$, but not $x$, form a $C_3^{(3)}$.

\textbf{Case 2:} $\{x,y\}\cap e=\emptyset$ and $e$ intersects only two bricks of $G_n$. Let $e=\{u,v,w\}$ where $u$ and $v$ are in the same brick. It can be readily verified that regardless of which vertices $u$ and $v$ are in this brick, there is an edge in $G_n$ that contains exactly one of $u$ or $v$ and one of $x$ or $y$, say that it contains $u$ and $x$, but not $v$. This edge, along with any edge containing $x$ and $w$, and $e$ form a $C_3^{(3)}$ in $G_n$.

\textbf{Case 3:} $\{x,y\}\cap e=\emptyset$ and $e$ intersects three distinct bricks. Let $u,v\in e$. Then $e$ along with any edge containing $u$ and $x$, and any edge containing $v$ and $x$, form a $C_3^{(3)}$ in $G$.

Thus, in all cases, we find a $C_3^{(3)}$ in $G_n+e$, so $G_n$ is saturated.
\end{proof}

	\section{Lower Bound: Proof Sketch}\label{section lower bound proof sketch}
	
		We will assume for the rest of the paper that $n$ is large, and that $G$ is a $C_3^{(3)}$-saturated $3$-uniform hypergraph on $n$ vertices with $\mathrm{sat}_3(n,C_3^{(3)})$ edges. By Theorem~\ref{theorem upper bound}, we can crudely assume
		\[
		|E(G)|\leq 2n.
		\]
		
		At the simplest level, the goal of our proof will be to show that the average degree of $G$ is at least $4-o(1)$. To do this, we will use a  discharging scheme: every vertex will start with charge equal to its degree, and we will then move charge around according to certain rules until the following two things are satisfied:
		\begin{itemize}
			\item Every vertex has non-negative charge, and
			\item $n-o(n)$ vertices have charge at least $4$.
		\end{itemize} 
		In particular, charge will always be moved around via edges, i.e. any time our discharging rules move charge from one vertex to another, they will be adjacent.
		
		Let $\ell=\ell(n)\ll \log(n)$ be a function that tends to infinity with $n$ (but slowly). Let
		\[
		L=\{v\in V(G)\mid d(v)< \ell\},
		\]
		and
		\[
		M=V(G)\setminus L.
		\]
		We will call  the vertices in $L$ \emph{low vertices} of $G$, and the vertices in $M$  \emph{non-low vertices}. Since $|E(G)|\leq 2n$, 
		we have 
		\[
		|M|\leq \frac{3|E(G)|}{\ell}\leq \frac{6n}{\ell}=o(n).
		\]
		Since there are few vertices in $M$, they can give all their charge to vertices in $L$. We also will need to keep track of vertices whose degree is almost linear in $n$. In particular, let 
		\[
		H=\{v\in V(G)\mid d(v)\geq  n/\ell^2\}.
		\]
		We will call vertices in $H$  \emph{high vertices}. Note that for a vertex being non-low is a much weaker condition than  being high. 
		
		Furthermore, for $i\leq \ell$ we will denote by $L_i$ the set of vertices of degree at most $i$. So, $L_{\ell-1}=L$. We will say a vertex $v\in L$ is $i$-{\em flat} if  the total number of vertices in $M$ in all edges containing $v$ (counted with multiplicities) is exactly $i$, and $v\in L$ is $i^+$-flat if $v$ is $j$-flat for some $j\geq i$.
		
		One of the first results we will prove (Lemma~\ref{lemma x and y in H}) implies that there is a set of two vertices, $x$ and $y$ such that almost all vertices in $L$ that are not adjacent to at least two vertices in $H$ are in $N(\{x,y\})$. Among other things, this implies that almost all vertices in $L$ are adjacent to at least one vertex in $H$, which will be very helpful.
		
		\subsection{Simple discharging rules that do not work}
		
		To motivate the rest of the proof, it is useful to mention a simple discharging scheme that does not work, but  is the basis for the more complicated discharging rules we use.
		
		The following simple discharging rules  do not create or destroy charge, they simply move it (we add an asterisk to the label for these rules because they are not the final discharging rules we use, and are instead simply a heuristic to motivate our final rules):
		
		\begin{itemize}
			\item[(D1*)] Every $(M,L,L)$ edge $\{h,a,b\}$ with $h\in M$ removes  charge $1$ from $h$ and gives charge  $1/2$  to each of $a$ and $b$.
			\item[(D2*)] Every $(M,M,L)$ edge $\{h_1,h_2,a\}$ with $a\in L$ removes charge  $1$ from each of $h_1$ and $h_2$, and gives charge $2$ to $a$.
		\end{itemize}
	
		Under this discharging scheme, every vertex in $M$ ends up with non-negative charge while vertices in $L$ only receive extra charge beyond the initial charge they had from their degree, making them more likely to end up above charge $4$. However, it is possible that many vertices in $L$ do not get up to charge $4$. 
		
		Now by this simple discharging scheme, every $i$-flat vertex in $L$ gets extra charge at least $i/2$ from vertices in $M$, but without more information, we cannot guarantee a vertex gets any more than this. Aside from $0$-flat vertices (which we already know there are few of by Lemma~\ref{lemma x and y in H}), this  scheme leaves the following types of vertices in $L$ below charge $4$:
		
		\begin{itemize}
			\item $1$-flat  $3$-vertices,
			\item $1$-flat  $2$-vertices,
			\item $2$-flat  $2$-vertices, and
			\item  $1$-vertices and  $0$-vertices.
		\end{itemize}
		We will show that there are very few $1$-vertices and $0$-vertices. However, it is possible that $G$ contains $\Omega(n)$ vertices of any of the other three types listed above. To deal with this, we introduce more complicated discharging rules, both refining the rule (D1*) (i.e. still moving charge from vertices in $M$ to vertices in $L$, but possibly not splitting charge evenly among the low vertices in an $(M,L,L)$ edge), as well as introducing some rules which will move charge from vertices in $L$ that have excess charge to other vertices in $L$ that are not satisfied by their non-low neighbors.
		
		\subsection{Overview of how to deal with \texorpdfstring{$1$}{2}-flat \texorpdfstring{$3$}{2}-vertices}
		
		We wish to enact a discharging rule such as
		\begin{itemize}
			\item[(D1.1*)] Every $(M,L,3)$ edge $\{h,a,b\}$ with $h\in H$, $a\in L$ and $b$ a $1$-flat  $3$-vertex removes charge $1$ from $h$ and gives charge $1$ to $b$.
		\end{itemize}
	
		If this rule were well-defined, then every $1$-flat $3$-vertex would receive an extra charge  $1$, bringing them up to charge $4$. Thus, the goal of Section~\ref{section 3 vertex} will be to show that a rule very similar to this is well-defined, and that the rule does not cause any other vertices to end up with too little charge. In particular, we will establish claims that essentially imply the following:
		\begin{itemize}
			\item There are very few $(M,3,2)$ edges,
			\item There are very few $(M,3,3)$ edges containing either two $1$-flat  $3$-vertices or a $1$-flat  $3$-vertex and a $2$-flat  $3$-vertex, and
			\item There are very few $3$-flat $3$-vertices that are in two or more $(M,3,3)$ edges with $1$-flat $3$-vertices.
		\end{itemize}
		Together these statements will imply that for almost all $1$-flat $3$-vertices, a rule like (D1.1*) will leave them with charge $4$, while not having many $2^+$-flat $3$-vertices end up with charge less than $4$.
		
		\subsection{Overview of how to deal with  \texorpdfstring{$2$}{2}-vertices}
		
		The simple discharging scheme presented above leaves both $1$-flat and $2$-flat $2$-vertices with charge less than $4$. We again wish to refine the rule (D1*) to prioritize giving  $2$-vertices more charge. For example, we wish to enact a discharging rule similar to
		\begin{itemize}
			\item[(D1.2*)] Every $(M,L,2)$ edge $\{h,a,b\}$ with $h\in M$, $a\in L$, and $d(b)=2$ removes charge $1$ from $h$ and gives charge $1$ to $b$.
		\end{itemize}
		As with rule (D1.1*), it is not obvious that this rule is well-defined, and also it is possible this rule conflicts with (D1.1*). However we will show that there are very few $(M,2,2)$ or $(M,3,2)$ edges, which implies that for almost all  $2$-vertices, (D1.2*) is well-defined and does not conflict with (D1.1*).
		
		(D1.2*) is enough for $2$-flat $2$-vertices to get charge $4$, but $1$-flat $2$-vertices would still be left with only  charge $3$. To deal with this, we  need to move charge from vertices in $L$ to $1$-flat $2$-vertices. In Section~\ref{section 2 vertex}, we will define the concept of a low vertex being ``helpful'', which essentially will imply that these low vertices have enough charge to give $1/2$  to each of their $1$-flat $2$-neighbors (with multiplicity), while still being satisfied. This will allow us to enact a discharging rule like
		\begin{itemize}
			\item[(D3*)] Every $(V,L,2)$ edge $\{z,a,b\}$ with $a\in L$ being a ``helpful'' vertex and $b$ being a $1$-flat $2$-vertex will remove charge $1/2$  from $a$ and give charge $1/2$  to $b$.
		\end{itemize}
		If we could show that almost every edge containing a $1$-flat $2$-vertex also contained a helpful vertex, (D3*) would be enough to satisfy almost all $1$-flat $2$-vertices, since they would get an extra charge $1/2$ from each of their two edges, bringing them up to charge $4$. Unfortunately, it is unclear if this is true. 
		
		To deal with this, we use a second round of moving charge from vertices in $L$ to $1$-flat  $2$-vertices, this time asking for a little less from the ``helping'' vertex. In particular, we define ``half-helpful'' vertices, which essentially are vertices in $L$ that may not have had enough charge to be ``helpful'', but can give $1/2$ to all the $1$-flat $2$-neighbors that weren't already satisfied by the charge moved via (D3*). This leads us to a rule like
		\begin{itemize}
			\item[(D4*)] Every $(V(G),L,2)$ edge $\{z,a,b\}$ with $a\in L$ being ``half-helpful'', and $b$ being a $1$-flat $2$-vertex that still does not have charge $4$ after the implementation of all previous rules, takes charge $1/2$ from $a$ and gives charge $1/2$ to $b$.
		\end{itemize}
	
		In order to make sure that (D3*) and (D4*) are well-defined and satisfy almost all  $2$-vertices, we will need to prove claims that essentially imply the following:
		\begin{itemize}
			\item There are almost no $(L,2,2)$ edges, and
			\item Every $(V,L,2)$ edge that contains a $1$-flat $2$-vertex also contains a ``helpful'' or ''half-helpful'' vertex.
		\end{itemize}
	All of this is done in Section~\ref{section 2 vertex}.
		
		\subsection{One More Discharging Rule}
	
	
	In order to guarantee that we have enough ``helpful'' and ``half-helpful'' vertices in $G$, we will further refine (D1*) to move charge away from vertices that we expect to have a lot of extra charge (say vertices of degree more than $8$) to other vertices. We will classify $(M,L,L)$ edges as ``rich'' and name one of the low vertices in every ``rich'' edge as the ``recipient'' if we expect the second vertex of this ``rich'' edge to be ``helpful'' even without charge from this edge (for example if a vertex is degree $8$ or more). Thus, we will enact a rule like:
	
		\begin{itemize}
			\item[(D1.3*)] Every ``rich'' $(M,L,L)$ edge $\{h,u,v\}$ with $h\in H$ and recipient $v$ removes  charge $1$ from $h$ and gives  charge $1$ to $v$.
		\end{itemize}
		
		These heuristic discharging rules capture most of the main ideas of the proof. In Section~\ref{section discharging}, we go through the actual discharging scheme we use.
		
		The final main idea in our proof that has not been mentioned here is the use of ``garbage sets'', i.e. small sets of vertices with numerous ``bad'' properties, which we will show we can largely ignore. In particular, throughout this heuristic section we have used the phrase ``almost all'' quite loosely. In order to rigorously show that $n-o(n)$ vertices end up with charge at least $4$, we will define a sequence of $10$ ``garbage sets'', $R_1, R_2,\dots, R_{10}$. We will show that the union of these sets is $o(n)$ and that all vertex in $L$ outside of these sets indeed end up with charge at least $4$  after the discharging rules take place.
		
		As we define the ``garbage sets'' $R_i$, we often want to exclude not just the vertices in $R_i$, but low vertices that are neighbors of $R_i$ as well. Some of these ``garbage sets'' could be around size $n/\ell$, so we may not be able to include all low neighbors if a ``garbage set'' contains any vertices of degree close to $\ell$. We will however usually include all low neighbors of the vertices of degree at most $8$ inside a ``garbage set''. To that end, given a garbage set $R_i$, we will usually define a set $\mathbf{R}_i$ that contains all the previously defined garbage sets, and then a set $\mathbf{R}_i'$ which contains $\mathbf{R}_i$, along with any vertices in $L$ that have a neighbor of degree at most $8$ in $\mathbf{R}_i$.

	\section{Lower Bound: Preliminary Lemmas}\label{section preliminary lemmas}
	
		\begin{lemma}\label{lemma x and y in H}
		Let $Q\subseteq N(H)$ denote the set of  vertices which have at least two neighbors in $H$ and let $S=V(G)\setminus (Q\cup H)$. There exist two vertices, $x,y\in V(G)$ with $x\in H$ such that 
		\[
		|S\setminus N(\{x,y\})|\leq 20 n/\ell.
		\]
	\end{lemma}
	
	\begin{proof}
		We will prove a slightly stronger statement, which implies our result. Namely, we will prove the following: If $|S|\geq 20n/\ell$, then there either exists
		\begin{itemize}
			\item a pair of vertices $x,y\in H$ such that $|S\setminus N(\{x,y\})|\leq 4n/\ell$, or
			\item a single vertex $x\in H$ such that $|S\setminus N(\{x\})|\leq 8n/\ell$.
		\end{itemize}
		
		Let $\delta:=4/\ell$. We claim that there are no partitions $S=S_1\cup S_2\cup S_3$ and $H=H_1\cup H_2\cup H_3$ such that $|S_1|\geq|S_2|\geq |S_3|\geq \delta n$, and such that $S_i\cap N(H\setminus H_i)=\emptyset$ for $i\in [3]$.
		
		Indeed, if we consider the $|S_1|\cdot|S_2|\cdot|S_3|:=s\geq \delta^3n^3=\omega(n^2)$ triples that contain one vertex in each set $S_i$, $s-2n=(1+o(1))s$ of these must be non-edges, and thus must contain a good pair. Let $uv$ be one such good pair. Then $uv$ can cover at most $\max\{|S_1|,|S_2|,|S_3|\}=|S_1|$ of the $(1+o(1))s$ non-edges, so there must be at least 
		\begin{equation}\label{equation good edge lower bound}
			(1+o(1))|S_2||S_3|\geq (1+o(1))\delta^2 n^2
		\end{equation}
		good pairs with endpoints in different $S_i$'s. Note that no vertices in $H$ can serve as the center of a $u,v$-link connecting any of these good pairs since the pairs contains vertices in two different $S_i$'s, and since each other vertex $v$ can be the center of a link including at most $4\binom{d(v)}2$ other vertices, we have that there are at most
		\begin{equation}\label{equation good edge upper bound}
			\sum_{v\in V(G)\setminus H}4\binom{d(v)}2\leq \frac{2}{\ell^2} n \sum_{v\in V(G)\setminus H} d(v)\leq \frac{6}{\ell^2} n\cdot|E(G)|\leq \frac{12}{\ell^2} n^2,
		\end{equation}
		such good pairs, where the last inequality follows from Theorem~\ref{theorem upper bound}.  However, this and the value of $\delta$ contradict \eqref{equation good edge lower bound}, so no such partition could exist.
		
		Now, if there exist two vertices $x,y\in H$ with $|N(x)\cap S|,|N(y)\cap S|\geq \delta n$, then 
		\[
		|S\setminus N(\{x,y\})|<\delta n,
		\]
		since otherwise $H=\{x\}\cup\{y\}\cup (H\setminus\{x,y\})$ would give us a partition that we know does not exist, so in this case we are done. 
		
		If there is exactly one vertex $x\in H$ with $|N(x)\cap S|\geq \delta n$, then $|S\setminus N(x)|\leq 3\delta n$, since otherwise we could partition $H\setminus \{x\}$ into two sets $A_1,A_2$ such that $\delta n\leq |N(A_1)\cap S|\leq 2\delta n$ by starting with all the vertex of $H\setminus\{x\}$ in $A_2$, and then moving them one at a time into $A_1$ until the first time the inequality is satisfied. This further implies that $|N(A_2)\cap S|\geq \delta n$, so $H=\{x\}\cup A_1\cup A_2$ would be a partition that we know does not exist. 
		
		Finally, if no vertices in $H$ are adjacent to at least $\delta n$ vertices in $S$, then we can  find a partition $H=A_1\cup A_2\cup A_3$ with $\delta n\leq |A_i|\leq 2\delta n$ for $i=1,2$ and consequently $|A_3|\geq \delta n$, again by starting with all the vertex in $H$ in $A_3$, moving them one at a time into $A_1$ until the first inequality is satisfied, then one at a time into $A_2$ until the second is satisfied, giving us a partition we know does not exist.
	\end{proof}
	
	Let $\mathbf{R}_1=R_1$ be the set of vertices of degree at most $8$ in the set $S\setminus N(\{x,y\})$ as defined in Lemma \ref{lemma x and y in H}. Let  $\mathbf{R}_1'= L\cap N[\mathbf{R}_1]$.
	
	\begin{remark}\label{remark size of R1}
		Since $\mathbf{R}_1\subseteq L_8$,  Lemma~\ref{lemma x and y in H} yields
		$
		|\mathbf{R}'_1|\leq 17 |\mathbf{R}_1|\leq 340 n/\ell.
		$
	\end{remark}
	
	
	We now focus on vertices that are in pairs with high codegree relative to their degree.
	
	\begin{claim}\label{claim degree d d neighbor}
	   Let $u,v\in V(G)$. If $d(uv)=d(v)\leq |V(G)|-3$, then $d(u)\geq d(v)+2$. Furthermore, if $d(uv)=d(v)=2$, then $v$ is the only degree $2$ double neighbor  of $u$.
	\end{claim}
	
	\begin{proof} 
		First,  assume  that $d(uv)=d(v)=d$ while $d(u)\leq d+1$. If $d(u)=d$, then for any $w\in V(G)\setminus N(u)$ the non-edge $\{u,v,w\}$ intersects every edge containing $u$ or $v$ in two vertices, so $G+\{u,v,w\}$ does not contain a $C_3^{(3)}$, a contradiction. Thus, we can assume $d(u)=d+1$. 
		
		Suppose $\{u,v,a_1\},\{u,v,a_2\},\dots,\{u,v,a_d\}\in E(G)$, and let $\{u,b_1,b_2\}$ be the single edge that contains $u$ but not $v$. If $b_1\not\in \{a_1,\dots,a_d\}$, then consider the non-edge $\{u,v,b_1\}$. This non-edge intersects every edge containing $u$ or $v$ in two vertices, so again $G+\{u,v,b_1\}$ is $C_3^{(3)}$-free, again a contradiction. Thus, we may assume $b_1\in\{a_1,\dots,a_d\}$, and similarly $b_2\in \{a_1,\dots,a_d\}$.
		
		Now consider the non-edge $\{v,b_1,b_2\}$ and let $T$ be the $C_3^{(3)}$ in $G+\{v,b_1,b_2\}$. A  $b_1,b_2$-link that avoids $v$ would  also avoid $u$ since $v$ is in every edge that $u$ is in except $\{u,b_1,b_2\}$, so $\{b_1,b_2\}$ cannot be 
		the set of core vertices of $T$,  thus $v$ is a core vertex. Without loss of generality, assume $b_1$ is the second core vertex. Then any $v,b_1$-link must use some edge $\{u,v,a_i\}$ where $a_1\not\in\{b_1,b_2\}$, and  a second edge $\{a_i,b_1,z\}$ for some $z\not\in \{u,v,b_1,b_2,a_i\}$. But then $\{u,b_1,b_2\}$, $\{u,v,a_i\}$ and $\{a_i,b_1,z\}$ form a $C_3^{(3)}$ in $G$, a contradiction. Thus, $d(u)\geq d(v)+2$.
		
		Now  assume that $d(v)=2$ and  that there exists a vertex $w$ with $d(uw)=d(w)=2$ as well. If $\{u,v,w\}\not\in E(G)$, then adding it cannot create a $C_3^{(3)}$ since this non-edge intersects all edges containing $v$ and $w$ in two vertices. Thus, we may assume $\{u,v,w\}$ is an edge of $G$.
		
		Let $\{u,v,v'\}$ and $\{u,w,w'\}$ be the other edges containing $v$ and $w$, respectively. If $v'=w'$, then the non-edge $\{v,w,v'\}$ again intersects all edges containing either $v$ or $w$ in two vertices, another contradiction. Thus we have  $v'\neq w'$.
		
		Now consider the non-edge $\{u,v,w'\}$. As this non-edge intersects every edge containing $v$ in two vertices, there 
		is a $u,w'$-link that avoids $v$. To avoid a $C_3^{(3)}$ in $G$ with the edge $\{u,w,w'\}$, $w$ must be contained in this link,
		 but every edge containing $w$ intersects $\{u,v,w'\}$ in two vertices, which is a contradiction. This proves the claim.
	\end{proof}
	
	\begin{claim} \label{Claim 2 deg 2}
		Let $v\in V(G)$. No edge $e\in E(G)$ not containing $v$ can have two $2$-vertices in $N(v)$. 
	\end{claim}
	
	\begin{proof} Suppose to the contrary that $e=\{u_1,u_2,u_3\}\in E(G)$, where $u_1,u_2\in N(v)$ and $d(u_1)=d(u_2)=2$. By definition, for $i=1,2,$ there is an edge $f_i=\{v,u_i,w_i\}$. By Claim~\ref{claim degree d d neighbor}, $w_i\neq u_{3-i}$ since  $2$-vertices cannot be double neighbors with each other. 
		
		Let $g=\{v,u_1,u_2\}$. Since $d(u_1)=2$  and  $u_1\in e\cap f_1$, $g\not\in E(G)$. Thus $G+g$ has a copy of $C_3^{(3)}$, $\mathcal{T}$ that contains the edge $g$. The edge $g$ intersects all edges containing $u_1$ and $u_2$ in two vertices, so neither of these vertices can be core vertices in $\mathcal{T}$, a contradiction.
	\end{proof}
	
	\begin{claim}\label{claim good pair double neighbor}
		If $e=\{a,b,c\}$ is an edge in $G$, and  $ab$ is a good pair, then $c$ is a double neighbor  of  $a$ or $b$.
	\end{claim}
	
	\begin{proof}
		If $ab$ is a good pair, then there is an $a,b$-link. If this link does not contain $c$, then together with $e$ they form a $C_3^{(3)}$ in $G$, a contradiction. Thus, one of the edges of the loose path contains $c$, which implies $c$ is a double neighbor  of  $a$ or $b$.
	\end{proof}

	Let $R_2$ be the set of vertices $z\in L_8\setminus \mathbf{R}'_1$ such that there exists some $h\in M$ where $zh$ is a bad pair and $N(z)\setminus \{h\}\subseteq L$. Let $\mathbf{R}_2=R_2\cup \mathbf{R}_1' $, and  let 
	$\mathbf{R}'_2= \mathbf{R}_2\cup (L\cap N(L_8\cap \mathbf{R}_2)) $.


\begin{lemma}\label{lemma few bad pairs with low degrees in G-h}
   $|R_2|\leq 12{\ell^2}$. 
\end{lemma}

\begin{proof}
   For a vertex $h\in M$, let $R_2(h)$ be the set of vertices $z\in L_8\setminus \mathbf{R}'_1$ such that  $zh$ is a bad pair and $N(z)\setminus \{h\}\subseteq L$. For $z\in R_2(h)$,
    the condition that $N(z)\setminus \{h\}\subseteq L$ implies that $h$ is the only vertex in $M$ adjacent to $z$, and further the condition that $z\not\in \mathbf{R}_1$ implies that $h\in \{x,y\}$. Thus $R_2=R_2(x)\cup R_2(y)$.
   
   Fix some $h\in \{x,y\}$ and let $z\in R_2(h)$.
   Let $A_1=N(z)\setminus \{h\}$. Note that $|A_1|\leq 2\ell$ since $z\in L$. Let $A_2=N(A_1)$. Since $A_1\subseteq L$, $|A_2|\leq (2\ell+1)|A_1|\leq 6\ell^2$. Then if $|R_2(h)|\geq 6\ell^2+1$, there exists some $z'\in R_2(h)\setminus A_2$. Adding the edge $\{h,z,z'\}$ gives a contradiction, since $h$ cannot be a core vertex because $hz$ and $hz'$ are bad pairs, and $z$ and $z'$ are too far apart in $G-h$ to be core vertices together. Thus the number of choices for $z$ adjacent to $h$ is at most $6\ell^2$.
   
  Since $R_2=R_2(x)\cup R_2(y)$, this yields $|R_2|\leq 12{\ell^2}$. 
\end{proof}

	Let $R_3$ be the set of $1$-flat vertices $u\in L_8\setminus \mathbf{R}'_2$  such that there exists an
	$(M,L,L)$ edge $\{h,u,v\}$ with $h\in M$  and $d(hv)=1$. 
Let $\mathbf{R}_3=R_3\cup \mathbf{R}_2' $, and  let 
	$\mathbf{R}'_3= \mathbf{R}_3\cup (L\cap N(L_8\cap \mathbf{R}_3)) $.


\begin{lemma}\label{no 1-1flat}
    $|R_3|\leq 16{\ell^2}$. 
\end{lemma}

\begin{proof} 
The only vertices in $M$ that are adjacent to  $1$-flat vertices outside of $\mathbf{R}'_2$ are $x$ and $y$. For $h\in \{x,y\}$, 
let  $\mathcal{L}(h)$ be the family of edges $\{h,u,v\}$ such that $u,v\in L\setminus \mathbf{R}'_2$, $u$ is a
$1$-flat vertex of degree at most $8$,   and $d(hv)=1$. We will show that $|\mathcal{L}(h)|\leq 8\ell^2$, which
 will imply that $|R_3|\leq 16{\ell^2}$. 

 Suppose to the contrary that $|\mathcal{L}(h)|>8\ell^2$ and let $e=\{h,u,v\}\in\mathcal{L}$. Let $A_1$ be the set of the vertices in $L\cap N[\{u,v\}]$, and $A_2=N[A_1]\setminus\{h\}$. Since $u,v\in L$, $|A_1|<4\ell$ and $|A_2|<8\ell^2$. Since $d(hu)=d(hv)=1$, each vertex adjacent to $h$ can be in at most one edge in $\mathcal{L}(h)$ in the role of $u$. Hence $A_2$ contains less than $8\ell^2$ vertices in the role of $u$ in $\mathcal{L}(h)$. Thus there is an edge $e'=\{h,u',v'\}\in \mathcal{L}(h)$ such that $u'$ is $1$-flat and is not in $A_2$.

Consider the non-edge $\{u,v,u'\}$. By Claim~\ref{claim good pair double neighbor}, since $d(uh)=d(vh)=1$, $uv$ is a bad pair, so there is a $u',u$-link avoiding $v$ or a $u',v$-link avoiding $u$. This link cannot contain $h$ since the only edge containing $h$ and one of $u$ or $v$ contains both. So since $u'$ is $1$-flat, this link must use a low edge containing $u'$, and some other edge connecting this to one of $u$ or $v$. But then $u'\in A_2$, a contradiction. 
Thus, $|R_3|\leq 16{\ell^2}$.
\end{proof}

	Let $R_4$ be the set of  vertices $u\in L_8\setminus \mathbf{R}'_3$  such that there exists an
	 $h\in M$  with $d(hu)=d(u)$.
Let $\mathbf{R}_4=R_4\cup \mathbf{R}_3' $, and  let 
	$\mathbf{R}'_4= \mathbf{R}_4\cup (L\cap N(L_8\cap \mathbf{R}_4)) $.

\begin{lemma}\label{lemma few i vertices that are i-neighbors of h}
 $|R_4|\leq 54\frac{n}{\ell}$. 
\end{lemma}

\begin{proof}
    Fix a vertex $h\in M$. Note that if $u$ and $u'$ are both vertices such that $d(hu)=d(u)$ and $d(hu')=d(u')$, then $\{h,u,u'\}\in E(G)$ since if not, adding $\{h,u,u'\}$ cannot create a $C_3^{(3)}$ in $G$ as it intersects every edge containing $u$ and $u'$ in at least two vertices. We claim that $h$ has at most $9$ neighbors $u$ with $d(u)\leq 8$ such that $d(hu)=d(u)$. Indeed, if $h$ had $10$ such neighbors, say $\{u_1,u_2,\dots,u_{10}\}$, then $u_{10}$ must be in $9$ edges, namely $\{h,u_1,u_{10}\}, \{h,u_2,u_{10}\}, \dots, \{h,u_9,u_{10}\}$, contradicting the fact that $d(u)\leq 8$.
    
    Since every vertex $h\in M$ has at most $9$ such neighbors, the total number of vertices $u$ such that $d(u)\leq 8$ and $d(hu)=d(u)$ for some $h\in M$ is at most
    \[
    9|M|\leq 9\cdot\frac{6n}{\ell}=54\frac{n}{\ell}.
    \]
\end{proof}

It is perhaps worth noting that all $0$-vertices and $1$-vertices end up in $\mathbf{R}_4'$ since they are either contained in $R_1$ or in $R_4$.

	\begin{lemma}\label{coneighbors}
 For each $h\in M$ and each  neighbor $u\in L\setminus \mathbf{R}'_4$  of $h$ with 
 $d(u)\leq 8$ and $d(hu)\geq 2$, one of the following holds:

\begin{enumerate}[label=(\alph*)]
\item  there is an edge $\{h,u,w\}$ such that   $d(w)\geq 4$ or $w$ is not $1$-flat and
$d(w)\geq 3$, or

\item there are two edges $\{h,u,w_1\}$ and $\{h,u,w_2\}$ such that neither of $w_1$ and $w_2$ is a $1$-flat $2$-vertex.
\end{enumerate}
\end{lemma}

\begin{proof} 
Suppose that for some  $u\in L_8\setminus \mathbf{R}'_4$   with $d(hu)\geq 2$  neither of (a) and (b) holds.
 This means that  there are  vertices $w_1,w_2$ such that for $i\in [2]$,  $e_i=\{h,u,w_i\}\in E(G)$ and at least one of them, say
$w_1$, is a $1$-flat $2$-vertex. Since $u\notin \mathbf{R}'_2$, $w_1\notin \mathbf{R}_2$, and hence
 $w_1h$ is a good pair.
Let edges $f_1$ and $f_2$ form a $w_1,h$-link with $w_1\in f_1$. In order to avoid a triangle with edges $e_1,f_1$ and $f_2$ in $G$,  $f_1$ or $f_2$ contains $u$.

{\bf Case 1:} $u\notin f_2$. Then $u\in f_1$, so that $f_1$ has the form $\{w_1,u,z\}$ and 
$f_2$ has the form $\{h,z,z'\}$ for some $z,z'$. If $w_2\notin \{z,z'\}$, then $G$ contains triangle with edges $f_1,f_2$ and
$\{h,z,z'\}$, a contradiction. Otherwise, the co-degree of $w_2h$ is at least $2$. Since  $u\notin \mathbf{R}'_4$, $w_2\notin \mathbf{R}_4$, and hence $w_2$
cannot be a double neighbor of $h$ of degree $2$.
 Thus Part (a) of the claim of the lemma holds.

{\bf Case 2:} $u\notin f_1$. Then $u\in f_2$, so that $f_2$ has the form $\{h,u,z\}$ and 
$f_1$ has the form $\{w_1,z,z'\}$ for some $z,z'$. By symmetry, we may assume $z=w_2$.
If $d(w_2)\geq 4$ or $w_2$ is a not $1$-flat $3$-vertex, then 
Part (a) of the lemma holds. 
 If $d(w_2)=2$, then the non-edge $f=\{h,w_1,w_2\}$ intersects each of the edges containing $w_1$ or $w_2$ in two vertices.
 Thus neither of $w_1$ or $w_2$ could be a core vertex in a triangle in $G+f$. This
 yields $d(w_2)\geq 3$. So, if (a) does not hold, then
\begin{equation}\label{c2}
 \mbox{\em $w_2$ is a $1$-flat $3$-vertex.   }   
\end{equation}
 Let $e_3=\{w_1,u,w_2\}$. Since we know both edges containing $w_1$, $e_3\notin E(G)$. So $G'=G+e_3$
 contains a triangle, say with edges $g_1,g_2$ and $e_3$. Since each of the edges containing $w_1$ has two common vertices with $e_3$, $g_1$ and $g_2$ form a $u,w_2$-link. We may assume that $u\in g_1$. 
 
 Since $G$ has no triangle with edges
 $g_1,g_2$ and $e_2$, $h\in g_1\cup g_2$. Since by~\eqref{c2} the co-degree of $hw_2$ is $1$,
 $g_1$ has the form $\{h,u,w_3\}$ and 
$g_2$ has the form $\{w_2,w_3,t\}$ for some $t$. If  $w_3$ is not a $1$-flat $2$-vertex, then Part (b) of the
 the lemma holds. So suppose $w_3$ is  a $1$-flat $2$-vertex.Then we know all edges containing at least one of $w_1,w_2$ and $w_3$.
 In particular, edge $e_4=\{w_1,w_2,w_3\}$ is not in $G$, and hence $G_1=G+e_4$ must contain a triangle. However, the edges containing at least one of $w_1,w_2$ and $w_3$ and sharing less than two vertices with $e_4$ are only
 $e_1,e_2$ and $g_1$. Any two of these edges share $u$ and $v$, a contradiction.

{\bf Case 3:} $u\in f_1\cap f_2$. We may assume that $f_1=\{w_1,u,z\}$ and $f_2=e_2$. In particular, $z\neq w_2$.
Consider $e_3$ and $G'$ as in Case 2. Again, there are edges $g_1$ and $g_2$ forming a $u,w_2$-link. 
Since $u\notin \mathbf{R}'_4$, $w_2\notin \mathbf{R}_4$; thus $w_2$ is 
 $1$-flat. Then $v\in g_1-g_2$, so we may assume $g_1=\{u,h,w_3\}$ and $g_2=\{w_2,w_3,t\}$ for some $t\notin \{h,u\}$. 
 So, we have Case 2 with $w_3$ in the role of $w_1$.
\end{proof} 

\subsection{Edges of type \texorpdfstring{$(M,L,2)$}{2}}
	
	We have two goals in this section. First, we want to prove a structural result which tells us essentially that in almost all $(M,L,2)$ edges, either such an edge contains a good pair involving the vertex in $M$, or the two low vertices in the edge are connected via a tight path of length $2$. Second, we will use this structural result to show that there are only few $(M,3,2)$ or $(M,2,2)$ edges in $G$.

	Let $R_5$ be the set of  vertices $u\in L_8\setminus \mathbf{R}'_4$  such that there exists an
	edge $\{h,u,t\}$ with the following properties
	\begin{itemize}
	\item $h\in M$ and $d(t)=2$,
		    \item $hu$ and $ht$ are bad pairs, and
		    \item there are no edges $\{u,a,b\}$ and $\{t,a,b\}$ for some $a,b\in V(G)\setminus \{h,u,t\}$.
		\end{itemize}
Let $\mathbf{R}_5=R_5\cup \mathbf{R}_4' $, and  let 
	$\mathbf{R}'_5= \mathbf{R}_5\cup (L\cap N(L_8\cap \mathbf{R}_5)) $.

	\begin{lemma}\label{Daras}
	$|R_5|\leq 2000\frac{n}{\ell}$. 
	\end{lemma}

	\begin{proof}
	    Fix $h\in M$. We claim that $h$ is in at most $314$ edges $\{h,u,t\}$ that satisfy the conditions in the  definition of $R_5$. To see this, assume to the contrary that $h$ is in at least $315$ such edges. Note that each such edge containing $h$ can intersect at most $8$ other edges in two vertices, possibly $7$ other such edges containing $hu$ and one containing $ht$. Thus, we can find a collection of $35=315/9$ such edges $\{h,u,t\}$ that only intersect in the vertex $h$.
	    
	Call these edges $e_1,e_2,\dots,e_{35}$, and let $e_i=\{h,u_i,t_i\}$, where $d(t_i)=2$. Furthermore, let $e_i'=\{t_i,a_i,b_i\}$ be the second edge containing $t_i$. Consider the non-edge $\{h,t_i,t_j\}$ for some $1\leq i,j\leq 35$, $i\neq j$. Since $ht_i$ and $ht_j$ are bad pairs, there is a $t_i,t_j$-link  that avoids $h$. Thus, this link  consists of the edges $e_i'$ and $e_j'$.
	
	Now, fix $t_1$, and let $U_a$ and $U_b$ denote the sets of vertices $t_j$ with $2\leq j\leq 35$ such that $e_i'\cap e_j'=\{a_1\}$ and $e_i'\cap e_j'=\{b_1\}$ respectively. One of $|U_a|$ or $|U_b|$ must have size at least $18$, so assume without loss of generality $|U_a|\geq 18$, and by reordering if necessary, we can assume that $a_i=a_j$ for $1\leq i,j\leq 18$.
	
	Now consider the non-edge $\{h,u_1,t_j\}$ for each $2\leq j\leq 18$. By assumption, $hu_1$ and $ht_j$ are bad pairs, so there is a  $u_1,t_j$-link that avoids $h$. This link must use the edge $\{t_j,a_j,b_j\}$. Note that the second edge of this link cannot be of the form $\{u_1,a_j,z\}$ since if $z=b_1$, then recalling that $a_j=a_1$, this gives us that $\{u_1,a_1,b_1\}$ and $\{t_1,a_1,b_1\}$ are edges, contradicting the initial assumption of the lemma, and if $z\neq b_1$, then the edges $\{h,u_1,t_1\}$, $\{u_1,a_j,z\}$ and $\{t_1,a_1,b_1\}$ form a $C_3^{(3)}$ in $G$. Thus, there must be an edge $\{u_1,b_j,z\}$ where $z\neq a_j$. So, $u_1$ is adjacent to $b_j$ for all $1\leq j\leq 18$.
	
	Now, since $e_i'\cap e_j'=\{a_1\}$ for all $2\leq i,j\leq 18$ with $i\neq j$, we have that $b_i\neq b_j$ for all $2\leq i,j\leq 18$, $i\neq j$. Thus, $|N(u_1)|\geq |\{b_2,b_3,\dots,b_{18}\}|=17$, which contradicts the fact that $d(u_1)\leq 8$ (and consequently, $|N(u_1)|\leq 16$.
	
	This establishes that indeed, every vertex $h\in M$ is in at most $314$ such edges. Thus, 
	\[|R_5|\leq 
	314\cdot |M|\leq 314\frac{6n}{\ell}<2000\frac{n}{\ell}.
	\]
	\end{proof}

We now use the above lemma to prove that there are very few $(M,2,2)$ and $(M,3,2)$ edges.

\begin{lemma}\label{lemma h22 h23 bad pairs}
	Let $h\in M$ and let $e_1=\{h,v_1,v_2\}$ be a $(h,2,2)$ or $(h,2,3)$ edge, where $v_1,v_2\notin \mathbf{R}'_5 $.
 Then $hv_1$ and $hv_2$ are bad pairs.
\end{lemma}

\begin{proof}
	By Claim~\ref{claim degree d d neighbor}, $v_1$ and $v_2$ cannot be double neighbors since one of them is degree $2$, and the other is degree at most $3$. 
	
	Assume  that $hv_1$ is a good pair. To avoid having a $C_3^{(3)}$ in $G$, any  $h,v_1$-link must contain $v_2$. Since $v_1$ is not a double neighbor  of $v_2$, this implies that we have edges $\{h,v_2,w\}$ and $\{v_1,w,z\}$ for some vertices $w,z\in V(G)\setminus \{h,v_1,v_2\}$. Since  $v_2$ is a double neighbor  of $h$ and is not in $\mathbf{R}_4$, we have $d(v_2)=3$, and consequently $d(v_1)=2$.
	
	Consider the non-edge $\{v_1,v_2,w\}$. If $d(v_2w)=2$, then this non-edge intersects every edge containing $v_1$ and $v_2$ in two vertices, so $G+\{v_1,v_2,w\}$ cannot contain a $C_3^{(3)}$. Thus, $d(v_2w)=1$. Furthermore, since $\{v_1,v_2,w\}$ intersects every edge containing $v_1$ in two vertices, there exists a $v_2,w$-link that avoids $v_1$. To avoid a $C_3^{(3)}$ in $G$, $h$ must be in this link. Since $v_2$ is not in $R_4$,
	 $d(hv_2)<d(v_2)=3$, so this $v_2,w$-link must contain edges of the form $\{h,w,a\}$ and $\{v_2,a,b\}$.  But  this creates a $C_3^{(3)}$ in $G$ with $\{h,v_1,v_2\}$, a contradiction. 
\end{proof}

\begin{lemma}\label{lemma h22 h23 house}
	Let $h\in M$ and let $\{h,v_1,v_2\}$ be a $(h,2,2)$ or a $(h,3,2)$ edge of $G$, where $v_1,v_2\notin \mathbf{R}'_5 $. Let $\{v_1,a,b\}$ and $\{v_2,a,b\}$ for some $a,b\in V(G)\setminus\{h,v_1,v_2\}$. Then at most one of $e_2$ or $e_3$ is in $E(G)$.
\end{lemma}

\begin{proof} 
	Suppose $G$ contains all the edges $\{h,v_1,v_2\}$, $\{v_1,a,b\}$ and $\{v_2,a,b\}$ for some $a,b\in V(G)\setminus\{h,v_1,v_2\}$. If $d(v_1)=d(v_2)=2$, then  the edge $\{v_1,v_2,a\}$ is not an edge of $G$, but it intersects every edge containing either $v_1$ or $v_2$ in two vertices. So adding this edge to $G$ does not create a $C_3^{(3)}$. Thus, we may assume that $\{h,v_1,v_2\}$ is a $(h,3,2)$ edge, say with $d(v_1)=2$ and $d(v_2)=3$. If $d(v_2a)=2$, then the non-edge $e'=\{v_1,v_2,a\}$ intersects every edge containing $v_1$ and $v_2$ in two vertices, so $G+e'$ cannot contain a $C_3^{(3)}$, a contradiction. Thus $d(av_2)=1$. Similarly, $d(bv_2)=1$.
	
	Suppose now that $d(hv_2)=2$, say $\{h,v_2,z\}\in E(G)$, and note that $z\not\in \{v_1,h,a,b\}$. Consider the non-edge $\{h,v_1,z\}$, and let $T$ be a $C_3^{(3)}$ in $G+\{h,v_1,z\}$. If $h$ and $z$ are core vertices in $T$, then any $h,z$-link
	 must contain $v_1$. But every edge that contains $v_1$ and one of $h$ or $z$ intersects $\{h,v_1,z\}$ in two vertices, so $h$ and $z$ cannot both be core vertices. Thus $v_1$ is a core vertex of $T$. So, $\{v_1,a,b\}$ is one of the edges of $T$. By Lemma~\ref{lemma h22 h23 bad pairs}, $hv_1$ is a bad pair, so $z$ must be the second core vertex of $T$. Thus, there exists some edge containing $z$ and exactly one of $a$ or $b$, say without loss of generality $\{z,a,w\}\in E(G)$. Note that $w\not\in \{h,v_1,v_2,b\}$, and thus $\{h,v_1,v_2\}$, $\{v_1,a,b\}$ and $\{h,a,w\}$ form a $C_3^{(3)}$ in $G$, a contradiction. Thus, $d(hv_2)=1$.
	
	Let $u_1,u_2\in V(G)$ be such that $\{v_2,u_1,u_2\}\in E(G)$, and note that $u_1,u_2\not\in \{h,v_1,a,b\}$. Consider the non-edge $\{v_1,u_1,u_2\}$. By Claim~\ref{claim good pair double neighbor}, $u_1u_2$ is a bad pair since $d(u_1v_2)=d(u_2v_2)=1$, and thus there is a $v_1,u_1$-link avoiding $u_2$ or a $v_1,u_2$-link avoiding $u_1$. Assume the former.
	 If this $v_1,u_1$-link uses the edge $\{h,v_1,v_2\}$, then it must also use an edge $\{h,u_1,w\}$ for some $w\not\in \{v_2,u_2\}$ since the only  edge containing $\{v_2,u_1\}$ is $\{v_2,u_1,u_2\}$. But then $\{h,v_1,v_2\}$, $\{h,u_1,w\}$ and $\{v_2,u_1,u_2\}$ form a $C_3^{(3)}$ in $G$, a contradiction. If instead, the link uses the edge $\{v_1,a,b\}$, then it must also use an edge containing exactly one of $a$ or $b$ and $u_1$, but not $u_2$, say  $\{a,u_1,w\}$ for some $w\not\in \{b,v_2,u_2\}$. Again we reach a contradiction since $\{v_2,a,b\}$, $\{v_2,u_1,u_2\}$ and $\{a,u_1,w\}$ form a $C_3^{(3)}$ in $G$. Thus, in all cases we arrive at a contradiction.
\end{proof}

Lemma~\ref{lemma h22 h23 house} together with the definition of $R_5$ yield the following.
\begin{corollary}
    \label{remM22}
For every $(h,2,2)$ or $(h,3,2)$ edge $\{h,v_1,v_2\}$ where $h\in M$, vertices $v_1$ and $v_2$ are in $\mathbf{R}'_5 $.
\end{corollary}
	
	\section{Lower Bound: \texorpdfstring{$1$}{2}-flat \texorpdfstring{$3$}{2}-vertices}\label{section 3 vertex}
	
		Let $R_6$ be the set of $2$-flat $3$-vertices  $u\in L\setminus \mathbf{R}'_5$  such that there exists an
	edge $\{h,u,v\}$ where $h\in M$ and $v$ is a $1$-flat $3$-vertex. Let $\mathbf{R}_6=R_6\cup \mathbf{R}_5' $, and  let 
	$\mathbf{R}'_6= \mathbf{R}_6\cup (L\cap N(L_8\cap \mathbf{R}_6)) $.

	\begin{lemma}\label{lemma few h33 edges with 1flat and 2flat}
	$|R_6|\leq 32{\ell}^2$. 
	\end{lemma}

	\begin{proof}
		For $h\in M$, let $\mathcal{F}(h)$ denote the set of $2$-flat $3$-vertices  $u\in R_6$ contained in some edge $\{h,u,v\}$
		where $v$ is a $1$-flat $3$-vertex. Since $v$ is a $1$-flat $3$-vertex and $u\notin \mathbf{R}'_1$, $v\notin \mathbf{R}_1$. So, $h\in \{x,y\}$. Assume  that $|\mathcal{F}(h)|\geq 16\ell^2+1$.
		
		We first describe a certain structure that all but $8\ell$ vertices in $\mathcal{F}(h)$ are contained in. Let $e=\{h,u,v\}$ be an edge of $G$ with $u\in\mathcal{F}(h)$ and $v$ a $1$-flat $3$-vertex. Since $u\notin \mathbf{R}'_3$, $v\notin \mathbf{R}_3$.
		Hence 
		 $d(\{u,h\})\geq 2$,  so by the case, $d(\{u,h\})=2$. Let $a$ be a vertex such that $\{h,u,a\}\in E(G)$. Since $u$ is $2$-flat, note that $a$ is low.

	Since $|N(u)\setminus \{h\}|\leq 4$ and every vertex in $N(u)\setminus \{h\}$ is low (and thus has at most $2\ell$ neighbors),
		$|N(N(u)\setminus\{h\})|\leq 8\ell.$
	Thus, if there are at least $8\ell+1$ vertices in $\mathcal{F}(h)$ that are in a bad pair with $h$, then we can find some $u'\in \mathcal{F}(h)\setminus N(N(u)\setminus\{h\})$ such that $hu'$ is also a bad pair. Consider the non-edge $\{h,u,u'\}$. Since $hu$ and $hu'$ are both bad pairs, there must be a  $u,u'$-link that avoids $h$. But since $u'\not\in N(N(u)\setminus\{h\})$, there is no such link,  a contradiction. Thus, there are at most $8\ell$ vertices in $\mathcal{F}(h)$ that are in a bad pair with $h$.
		
		Let $\mathcal{F}_{good}(h)$ denote the set of vertices in $\mathcal{F}(h)$ that are in a good pair with $h$. Then,
		\[
		|\mathcal{F}_{good}(h)|\geq |\mathcal{F}(h)|-8\ell\geq 8\ell^2+1.
		\]
		Assume now that $u\in\mathcal{F}_{good}(h)$, with $\{h,u,v\}$ and $\{h,u,a\}$ edges of $G$ with $v$ a $1$-flat $3$-vertex. Since $hu$ is a good pair, there exist edges $\{u,w_1,w_2\}$ and $\{w_2,w_3,h\}$ forming a  $u,h$-link. Note that to avoid a $C_3^{(3)}$ in $G$, we must have $v,a\in \{w_1,w_2,w_3\}$. Since $d(\{h,v\})=1$, we have  $v=w_1$. 
		
		If $a=w_2$, then consider the non-edge $\{h,v,a\}$. Note that any link between two of the vertices in $\{h,v,a\}$ that does not contain the third  cannot contain $u$ since every edge containing $u$ contains two vertices in $\{h,v,a\}$. Thus, if there were  such a link, say a $z_1,z_2$-link with $z_1,z_2\in \{h,v,a\}$, then there would be a $C_3^{(3)}$ in $G$ consisting of this link and the edge $\{u,z_1,z_2\}$,  a contradiction. Thus,  $a\neq w_2$, so $a=w_3$. For ease of notation, let us define $b:=w_2$. 
		
		Let $\{v,s,s'\}$ denote the third edge of $G$ containing $v$. We claim that $s,s'\not\in \{u,a,b\}$. Indeed, $u\not\in \{s,s'\}$ by Claim~\ref{claim degree d d neighbor}. If $a\in \{s,s'\}$, then the non-edge $\{u,v,a\}$ intersects every edge containing $u$ and $v$ in two vertices, and thus adding it to $G$ does not create a $C_3^{(3)}$. If $b\in \{s,s'\}$, then since $u,a\not\in \{s,s'\}$, the edges $\{h,u,v\}$, $\{v,s,s'\}$ and $\{h,a,b\}$ form a $C_3^{(3)}$ in $G$, a contradiction. Thus, $s,s'\not\in \{u,a,b\}$ as claimed. Note that this also implies that $uv$ is a bad pair since any  $u,v$-link would necessarily use the edges $\{h,u,a\}$ and $\{v,s,s'\}$, but $\{h,a\}\cap\{s,s'\}=\emptyset$.
		
		At this point, we define the \emph{configuration} of $u\in \mathcal{F}_{good}(h)$ to be the collection of vertices $\{u,v,a,b,z,z'\}$. We wish to speak of the configurations associated with multiple vertices from $\mathcal{F}_{good}(h)$ simultaneously, so given $u\in \mathcal{F}_{good}(h)$, we will let $v_u,a_u,b_u,s_u$ and $s'_u$ denote the vertices in the configuration of $u$ playing the roles of $v,a,b,z$ and $s'$ as described above. Note that while there is a unique choice of $v_u,a_u$ and $b_u$, the vertices $s_u$ and $s'_u$ are interchangeable, so we may arbitrarily choose which vertex in $N(v_u)\setminus\{h,u,b\}$ is $s_u$ and which is $s'_u$. We will label the configuration of $u$ as $\mathcal{C}_u$. We have proven the following about $\mathcal{C}_u$ for any $u\in\mathcal{F}_{good}(h)$:
		\begin{itemize}
			\item The edges $\{h,u,v_u\}, \{h,u,a_u\}, \{h,a_u,b_u\}, \{u,v_u,b_u\}$ and $\{v_u,s_u,s'_u\}$ are all present in $G$,
			\item The vertices $h,u,v_u,a_u,b_u,s_u$ and $s'_u$ are all distinct, and
			\item The pair $uv_u$ is a bad pair.
		\end{itemize}
		
		Now, given a vertex $w\in L$ and a vertex $u\in \mathcal{F}_{good}(h)$, note that $w$ is adjacent in $G-h$ to a vertex in $\mathcal{C}_u$ if and only if either $u$ or $v_u$ are in $N[N(w)\cap L]$. Since a vertex can play the role of $u$ or $v_u$ in at most one configuration, this implies that $w$ is adjacent to at most $|N[N(w)\cap L]|\leq 4\ell^2$ configurations.
		
		We claim that for every $u\in\mathcal{F}_{good}(h)$, $\{h,s_u,s'_u\}\in E(G)$, and no other edges containing $h$ and one of $s_u$ or $s'_u$ are in $E(G)$. Indeed, first consider the case where $s_u\not\in N(h)$. Then $s_u$ is adjacent to at most $4\ell^2$ configurations, so there exists a $u'$ such that $s_u$ is not adjacent to $\mathcal{C}_{u'}$. Consider the non-edge $\{s_u,u',v_{u'}\}$. Since $u'v_{u'}$ is a bad pair, there must be a loose  path of length $2$ from a vertex in $\{u',v_{u'}\}$ to $s_u$, but since $s_u$ is not adjacent to any vertex in $\mathcal{C}_{u'}\cup \{h\}=N(\{u',v_{u'}\}$, this is a contradiction. Thus, $s_u\in N(h)$. Now consider the possibility that we have an edge $\{s_u,h,z\}$ for some $z\neq s'_u$. Then clearly $z\not\in\{u,v_u\}$, and so $\{h,s_u,z\}$, $\{v_u,s_u,s'_u\}$ and $\{h,u,v_u\}$ form a $C_3^{(3)}$ in $G$, a contradiction. Thus, the only edge containing $h$ and $s_u$ is $\{h,s_u,s'_u\}$. By the symmetry between $s_u$ and $s_u'$, this also is the only edge containing $h$ and $s'_u$ as claimed.
		
		Now, given any $u\in\mathcal{F}_{good}(h)$, there are at most $8\ell^2<|\mathcal{F}_{good}(h)|$ configurations adjacent to either $s_u$ or $s'_u$ in $G-h$, so there exists some $u'\in \mathcal{F}_{good}(h)$ such that neither $s_u$ or $s_u'$ are adjacent to $\mathcal{C}_{u'}$. Consider the non-edge $\{s_u,s_u',u'\}$, and a triangle $T$ in $G+\{s_u,s_u',u'\}$. By the choice of $u'$, $u'$ cannot be a core vertex in $T$, and thus $s_us_u'$ is a good pair, but no edge of $G$ contains $h$ and exactly one of these vertices, so the  $s_u,s_u'$-link, along with the edge $\{s_u,s_u',h\}$ gives a $C_3^{(3)}$ in $G$, a contradiction.
		
		Thus for each $h\in \{x,y\}$, $|\mathcal{F}(h)|\leq 16\ell^2$. Therefore, $|R_6|\leq 32{\ell}^2$.
	\end{proof}
	
	We will call a vertex $v\in L$ \emph{supported} if $G$ contains a $(M,M,v)$ edge, and \emph{unsupported} otherwise. Let $R_7$ be the set of unsupported $3$-flat $3$-vertices  $u\in L\setminus \mathbf{R}'_6$ adjacent to at least two $1$-flat $3$-vertices. Let $\mathbf{R}_7=R_7\cup \mathbf{R}_6' $, and let 
	$\mathbf{R}'_7= \mathbf{R}_7\cup (L\cap N(L_8\cap \mathbf{R}_7))$.

	\begin{lemma}\label{lemma 3flat 3vertices are adjacent to only 1 1flat 3vertex}
			$|R_7|\leq 2n/\ell$.
	\end{lemma}
	
\begin{proof}	
 Suppose $|R_7|> 2n/\ell$. Let
 $u\in R_7$ and $v_1$ and $ v_2$ be $1$-flat $3$-neighbors of $u$. Since $u$ is unsupported, and is 
	contained in edges $e_1,e_2$ and $e_3$, where
	$e_i=\{u,v_i,h_i\}$,  $h_1,h_2,h_3\in M$ and $v_1$ and $v_2$ are $1$-flat $3$-vertices (possibly, some $h_i$s and/or $v_i$s coincide). In this case the set $U(u)=\{u,v_1,v_2,v_3\}$ will be called the {\em $u$-set}.
	
Since $v_1,v_2\notin \mathbf{R}_2$, $h_1,h_2\in \{x,y\}$ and the pairs $h_1v_1$ and $h_2v_2$ are good.
 Since $h_1v_1$ is a good pair, the codegree of either $v_1u$ or $h_1u$ is at least $2$. On the other hand, since
	$v_1$ is $1$-flat and $u$  is $3$-flat and unsupported, the codegree of  $v_1u$ is $1$. Thus  the codegree of  $h_1u$ is at least $2$.
	Similarly,  the codegree of  $h_2u$ is at least $2$. It follows that $h_1=h_2$.
	If also $h_3=h_1$, then $d(u)=d(h_1u)$, contradicting the fact that $u\notin \mathbf{R}_4$.

	Let $X$ be the set of $u\in R_7$ such that $h_1=x$.
	Since $|R_7|> 2n/\ell$, by  the symmetry between $x$ and $y$ we may assume that $|X|>n/\ell$.
	
	Let $u\in X$. Since
	$v_1$ and $v_2$ are $1$-flat, $v_3\notin \{v_1,v_2\}$. This implies that $ux$ is a bad pair. Since for every two distinct
	$u,u'\in \mathcal{Q}_2(x)$ we can try to add the edge $\{x,u,u'\}$ and the pairs $ux$ and $u'x$ are bad, the pairs
	$uu'$ are good in $G-x$. Since $v_1,v_2,v_3$ are low, the number of $u'\in X$ such that
	$h_3$ is adjacent to $u'$ is at least 
	$$|X|-3(2\ell)>n/\ell-6\ell.$$
	Denote the set of such $u'$ (including our $u$) by $X_1$.
	
	If there is an edge $f=\{x,h_3,w\}$ containing both $x$ and $h_3$, then there is a $u\in X_1$ such that $w$ is not in the $u$-set, and so $G$ has a triangle formed by $e_1,e_3$ and $f$, a contradiction. Thus $h_3\notin N(x)$. 
	
	Now, if for some $u\in X_1$ there is an edge $g=\{x,v_3,w\}$ containing $h_1$ and $v_3$, then again
	there is a $u'\in X_1$ such that $w$ is not in the $u'$-set, and so $G$ has a triangle formed by $e'_1,e'_3$ and $g$, a contradiction. Thus $v_3\notin N(x)$ for each $u\in X_1$.
	
	Consider again a $u\in X_1$. Since $\{x,v_1\}$ is a good pair, there are edges $g_1$ and $g_2$ forming a 
	$v_1,x$-link. In view of $e_1$, $u\in g_1\cup g_2$. Let $v_1\in g_1$. 
	Since the codegree of $v_1u$ is $1$, $u\in g_2-g_1$. It follows that
	$g_2=e_2$ and $v_2\in g_1$. So, we may assume $g_1=\{v_1,v_2,w\}$.
	
	We claim that
	\begin{equation}\label{44}
		w\notin N(x).
	\end{equation}
	Indeed, suppose $G$ has edge $e_4=\{x,w,w_1\}$. Since $v_1$ and $v_2$ are $1$-flat, $w_1\notin \{v_1,v_2\}$. Since codegree of $ux$ is two, $w_1\neq u$. But then $G$ has a triangle formed by $e_4,e_1$ and $g_1$, a contradiction. 
	This proves~\eqref{44}.
	
	Let $G'$ be obtained from $G$ by adding edge $e_5=\{x,v_1,v_3\}$. Then $G'$ must have a triangle $T$ formed by $e_5$ and some edges $f_1$ and $f_2$. If the core vertices of $T$ are $x$ and $v_1$, then
	in view of $e_1$, $u\in f_1\cup f_2$. Let $v_1\in f_1$. Since the codegree of $\{v_1,u\}$ is $1$, $u\in f_2-f_1$. 
	But both edges in $G$ containing $\{x,u\}$ have two common vertices with $e_5$, a contradiction. 
	If the core vertices of $T$ are $x$ and $v_2$, then we get a similar contradiction. Hence
	the core vertices of $T$ are $v_1$ and $v_2$. Let $v_1\in f_1$. 
	In view of $e_4$, $w\in f_1\cup f_2$. By the symmetry between $v_1$ and $v_2$, we may assume $w\in f_1$, say 
	$f_1=\{v_1,w,w_1\}$.
	
	We claim that
	\begin{equation}\label{45}
		w_1\notin N(x).
	\end{equation}
	Indeed, suppose $G$ has edge $e_6=\{x,w_1,w_2\}$. By~\eqref{44}, 
	$w_2\neq w$. Since we know all edges incident to $v_1$ and $u$,
	$w_2\notin \{v_1,u\}$. 
	But then $G$ has a triangle formed by $e_6,e_1$ and $f_1$, a contradiction. 
	This proves~\eqref{45}.
	
	By the definition of $f_1$ and $f_2$, $f_1\cap f_2$ is either $\{w\}$ or $\{w_1\}$.
	
	{\bf Case 1:} $f_1\cap f_2=\{w\}$, say $f_2=\{w,v_2,w_3\}$. Then we know all edges incident to $v_2$.
	Since $|X_1|>n/\ell -6\ell$, $v_1$ is not adjacent to $h_3$, and $w$ and $w_1$ are low vertices not adjacent to $x$ (by~\eqref{44}
	and~\eqref{45}), there exists $u'\in X_1$ such that $v_1'$ is at distance at least $3$ from $\{w,w_1\}$.
	
	Let $G'$ be obtained from $G$ by adding edge $e_7=\{w,w_1,v'_1\}$. Then $G'$ must have a triangle $T$ formed by $e_7$ and some edges $j_1$ and $j_2$. By the choice of $v'_1$, the core vertices of $T$ in $e_7$ are $w$ and $w_1$. Let $j_1$ contain $w_1$. In view of $f_1$, $v_1\in j_1\cup j_2$. We know all edges containing $v_1$, in particular, we see that the codegree of
	$\{w_1,v_1\}$ is $1$. Also,
	the only candidate for  $j_2$ is $g_1$. It follows that $j_2=g_1$ and $j_1\cap j_2=\{v_2\}$. But we also know all edges containing $v_2$, and none of them contains $w_1$, a contradiction.
	
	{\bf Case 2:} $f_1\cap f_2=\{w_1\}$, say $f_2=\{w_1,v_2,w_3\}$. Again,  we know all edges incident to $v_2$.
	Let $G'$ be obtained from $G$ by adding edge $e_8=\{v_1,w_1,v_2\}$. Then $G'$ must have a triangle $T$ formed by $e_8$ and some edges $j_1$ and $j_2$. 
	
	Suppose first that the core vertices of $T$ in $e_8$ are $v_1$ and $v_2$. Note that each of the edges $g_1,f_1$ and $f_2$ shares two vertices with $e_8$. Thus the only candidates for $j_1$ and $j_2$ are $e_1$ and $e_2$, but they have two common vertices, a contradiction.
	
	Suppose now that the core vertices of $T$ in $e_8$ are $v_1$ and $w_1$. Suppose that $v_1\in j_1$. Again, the only candidate for $j_1$ is $e_1$. On the other hand, in view of $f_1$, $w\in j_1\cup j_2$. Hence $w\in j_2-j_1$. But the third vertex of $j_2$ cannot be $x$ (by~\eqref{44}) and cannot be $u$ because we know all neighbors of $u$ and $w$ is not in this list.
	
	Finally, suppose now that the core vertices of $T$ in $e_8$ are $v_2$ and $w_1$. Suppose that $v_2\in j_1$. Now
	the only candidate for $j_1$ is $e_2$. Similarly to the previous paragraph,  in view of $f_2$, $w_3\in j_1\cup j_2$, and
	hence $w_3\in j_2-j_1$. Again,  the third vertex of $j_2$ cannot be $x$ (now by~\eqref{45}) and cannot be $u$ because we know all neighbors of $u$, and $w_1$ is not in this list. This finishes the proof of the lemma.
	\end{proof}

	\section{Lower Bound: Vertices of degree \texorpdfstring{$2$}{2} }\label{section 2 vertex}
	
	The goal of this section is to provide the results necessary to show that almost all $2$-vertices end up with charge at least $4$. Note that $3^+$-flat $2$-vertices are automatically supported, and thus will receive charge at least $4$, and $2$-flat $2$-vertices that are not supported should receive charge $1$  from each of their non-low neighbors, again leaving them satisfied. Thus, our main obstacle is $1$-flat $2$-vertices.
	
	As a reminder to the reader, we expect a $1$-flat $2$-vertex $t$ to get charge $1$  from its non-low neighbor, and 
	charge $1/2$  from two of its low neighbors, in particular, one low neighbor in each edge containing $t$ (possibly the same vertex twice if $t$ has a double neighbor ). One result which will be helpful is that there are almost no edges containing two $2$-vertices. Indeed, from
	Corollary~\ref{remM22} we already know each $(M,2,2)$ edge has two vertices in $\mathbf{R}'_5 $. Now we will show there are almost no $(L,2,2)$ edges.
	
			\subsection{Edges of type \texorpdfstring{$(L,2,2)$}{2}}
	
For $j\geq 2$, we will say that a vertex $u$ with $d(u)\leq j$ is a {\em $j$-far neighbor of a vertex $v$} if $u\in N(v)$, and
	
	\begin{enumerate} 
		\item for each edge $\{u,u',u''\}$ containing $u$ either $v\in \{u',u''\}$ or $\{u',u''\}\cap N(v)=\emptyset$; and\label{Property jfar neighborhood of v}
		\item for each edge $e\in E(G-v)$ containing $u$, the degree in $G$ of each vertex of $e$ is at most $j$\label{Property jfar degree}
	\end{enumerate}
	
Let $R_8$ be the set of $\ell$-far neighbors $u\in L\setminus \mathbf{R}'_7$     of vertices in $H$. Let $\mathbf{R}_8=R_8\cup \mathbf{R}_7' $, and  let 
	$\mathbf{R}'_8= \mathbf{R}_8\cup (L\cap N(L_8\cap \mathbf{R}_8)) $.

	\begin{claim}\label{claim far}
		For each $j\geq 2$ every $v\in V(G)$ has at most $2j^2$ $j$-far neighbors. Consequently,
		\[
		|R_8|\leq 2\ell^2|H|.
		\]
	\end{claim}
	
	\begin{proof} 
		Suppose, $v$ has at least $1+2j^2$ $j$-far neighbors, and $u$ is one of them. Let $u$ be in $k$ edges of the form $\{u,v,a\}$, where $a$ is also a $j$-far neighbor of $v$. Note that by Property \eqref{Property jfar neighborhood of v} in the definition of $j$-far, $u$ is only adjacent to other $j$-far neighbors of $v$ through edges that contain $v$. Since $d(u)\leq j$, $u$ has at most $2(j-k)$ neighbors in $G-v$. By Property \eqref{Property jfar neighborhood of v}, all these neighbors are not $j$-far neighbors of $v$, and again by Property \ref{Property jfar degree}, each of them is a neighbor of at most $j-1$ other $j$-far neighbors of $v$. There are at least
		\[
		(1+2j^2)-1-2(j-k)(j-1)=2j+2k j-2k>k
		\]
		$j$-far neighbors of $v$ at distance at least $3$ from $u$ in $G-v$, and since there are at most $k$ $j$-far neighbors of $v$ adjacent to $u$ in $G$, there exists at least one $j$-far neighbor of $v$ that is not adjacent to $u$ in $G$, and is distance at least $3$ from $u$ in $G-v$. Let $w$ be one such $j$-far neighbor of $v$.
		
		Let $e=\{v,u,w\}$ and $G+e$. Assume that $e$ together with edges $e'$ and $e''$ form a triangle $\mathcal{T}$ in $G+e$. If the core vertices of $\mathcal{T}$ in $e$   are $u$ and $w$, then $u$ and $w$ are at distance at most $2$ in $G-e$, a contradiction to the choice of $w$.
		
		So by the symmetry between $u$ and $w$, assume without loss of generality that the core vertices of $\mathcal{T}$ in $e$ are $u$ and $v$, and the edge $e'$ of $\mathcal{T}$ contains $u$. But then one of the vertices in $e'\setminus\{u\}$ is in $N(v)$, a contradiction to Property \eqref{Property jfar neighborhood of v} in the definition of $j$-far.
	\end{proof}
	
	\begin{lemma}\label{lemma no (L,2,2) edges}
	All	vertices of every $(L,2,2)$ edge are  in $\mathbf{R}'_8$.
	\end{lemma}
	
	\begin{proof}
		Let $e=\{a_1,a_2,b\}$, where $d(a_1)=d(a_2)=2$, and $d(b)< \ell$. Since not all vertices of $e$ are in $\mathbf{R}'_8$,
		 neither of $a_1$ and $a_2$ is in $\mathbf{R}_8$.
		Then $a_1$ and $a_2$ are not  in $\mathbf{R}'_1$ and hence are $1$-flat.
		 So, there are vertices $h_1,h_2\in H$ such that $h_1\in N(a_1)$ and $h_2\in N(a_2)$. By Claim~\ref{Claim 2 deg 2}, $h_1\not \in N(a_2)$, and  $h_2\not\in N(a_1)$. By Claim~\ref{claim degree d d neighbor}, $b$ cannot be a double neighbor  with both $a_1$ and $a_2$, so assume that $d(\{b,a_1\})=1$. 
		
		Since some vertex of $e$ is not in $\mathbf{R}'_8$,  $a_1$ cannot be a $\ell$-far neighbor of $h_1$, so $b\in N(h_1)$. Let $e_1=\{a_1,h_1,z_1\}$ be the other edge containing $a_1$ in $G$, and let $e_1'$ be an edge that contains $h_1$ and $b$. Since $a_2\not\in N(h_1)$, $a_2\not\in e_1'$, so for $e$, $e_1$, and $e_1'$ to not be a triangle in $G$, we must have that $z_1\in e_1'$. Let $e_2=\{a_2,h_2,z_2\}$ be the second edge containing $a_2$, and note that $z_2$ may be equal to $b$.
		
		Consider the edge $f=\{h_1,z_1,a_2\}$, and note that $f\not\in E(G)$ since $h_1$ and $a_2$ are not adjacent. Then $G+f$ contains a $C_3^{(3)}$, say ${T}$. Note that $h_1z_1$ is a bad pair by Claim \ref{claim good pair double neighbor} since $a_1$ is not a double neighbor  of either vertex. Thus, $a_2$ is one of the core vertices of ${T}$. This implies that either $e$ or $e_2$ is an edge of ${T}$. If $e$ is such edge, then $b$ must be in an edge with one of $h_1$ or $z_1$, but not both; this is a contradiction though as this edge, along with $e$ and $e_1$ would give us a $C_3^{(3)}$ in $G$. If $e_2$ is an edge of ${T}$, then we must have an edge $g$ in $G$ containing exactly one vertex from $\{h_1,z_1\}$ and exactly one vertex from $\{h_2,z_2\}$. Furthermore, $a_1\not\in g$ since $d(a_1)=2$ and so $g\neq e_1,e$. Consider two cases based on if $b=z_2$ or not.
		
		\textbf{Case 1:}  $b=z_2$. In this case, if $g$ contains $b$ and a vertex from $\{h_1,z_1\}$, this gives us a triangle in $G$ with edges $e,e_1$ and $g$. If $g$ contains $h_2$ and a  vertex from $\{h_1,z_1\}$, then this also gives us a triangle in $G$, this time with edges $e_1',e_2$ and $g$. Thus, we reach a contradiction.
		
		\textbf{Case 2:}  $z_2\neq b$. Note that since $a_2$ is not a $\ell$-far neighbor of $h_2$, $b\in N(h_2)$, we have some edge $e_2'$ containing both $b$ and $h_2$. Since in this case $a_2$ does not have any double neighbors, and $a_1\not\in e_2$ or $e_2'$, for the edges $e, e_2$ and $e_2'$ to not form a triangle in $G$, we must have that $z_2\in e_2'$. Now, since $|g\cap \{h_1,z_1\}|=|g\cap |\{h_2,z_2\}|=1$, the only possible way that $g, e_1'$ and $e_2'$ do not form a triangle is if $b\in g$, but in this case, $g,e_1$ and $e$ give us a triangle, and so we reach a contradiction.
	\end{proof}

	\subsection{Helpful and Half-Helpful Vertices}
	
	We now provide definitions which will help us deal with $2$-vertices.
	
	Given vertices $h\in M$ and $u\in L$, we will call the pair $hu$ \emph{rich} if $d(hu)\geq 3$, and we will call a vertex $u\in L$ \emph{rich} if $u$ is in any rich pair. We will call a $(M,L,4)$ edge $\{h,v,u\}$ with $h\in M$ and $d(u)=4$ \emph{exceptional} if
	\begin{itemize}
	    \item $d(hv)=2$ and $d(v)\geq 4$, and
	    \item $hu$ is a rich pair and $\mathrm{flat}(u)=2$.
	\end{itemize}
	We will call $v$ the \emph{exception} in the exceptional edge $\{h,v,u\}$.
	
	We now classify all $(M,L,L)$ edges of $G$ into three mutually exclusive types: needy edges, rich edges and reasonable edges.
	
	\subsubsection*{Needy Edges:} Let $e=\{h,u,v\}$ be a $(M,L,L)$ edge with $h\in M$. We will say $e$ is a \emph{needy edge with recipient $v$} if $v$ is either a $2$-vertex or a $1$-flat $3$-vertex, and $u$ is not.
	
	\subsubsection*{Rich Edges:} Let $e=\{h,u,v\}$ be a $(M,L,L)$ edge with $h\in M$ that is not needy. We will say $e$ is a \emph{rich edge with recipient $v$} if 
	\begin{enumerate}
	    \item $\{h,u,v\}$ is not exceptional,
	    \item $d(v)\leq 7$, $v$ is unsupported, and $hv$ is not rich, and
	    \item at least one of the following occurs:\begin{itemize}
	        \item $hu$ is a rich pair,
	        \item $d(u)\geq 3$ and $u$ is supported, or
	        \item $d(u)\geq 8$.
	    \end{itemize}
	\end{enumerate}
	
	\subsubsection*{Reasonable Edges:} We will call all $(M,L,L)$ edges that are not needy or rich, \emph{reasonable}.
	
	\subsubsection*{Helpful and Half-Helpful Vertices:}

	We now can turn our attention to the low vertices which will give charge to $1$-flat $2$-vertices. Given a low vertex $v$ let $\mathrm{rich}(v)$ denote the number of rich edges in which $v$ is the recipient, let $\mathrm{reas}(v)$ denote the number of reasonable edges containing $v$, and let $\mathrm{supp}(v)$ denote the number of $(M,M,v)$ edges, and finally let $\mathrm{flat}(v)$ be the number of edges containing $v$ and a $1$-flat $2$-vertex.
	
	Then, a low vertex $v$ of degree at least $3$ is \emph{helpful} if
	\begin{equation}\label{inequality helpful definition}
	    	d(v)+2\mathrm{supp}(v)+\mathrm{rich}(v)+\frac{1}{2}\mathrm{reas}(v)\geq \frac{1}2\mathrm{flat}(v)+4.
\end{equation}
	The left-hand side of~\eqref{inequality helpful definition} will be exactly the amount of charge $v$ ends up with before $v$ gives any charge away. Then, the right-hand side is the amount of charge we want $v$ to be able to give away, plus the amount of charge we want $v$ to keep (namely $4$). So, helpful vertices are exactly those low vertices that can give charge $1/2$  to each of their $1$-flat $2$-neighbors (counted with multiplicity), and still have charge at least $4$.
	
	Since (almost) every edge contains at most one $1$-flat $2$-vertex, it would be nice if we could simply show that every edge containing a $1$-flat $2$-vertex also contains a helpful vertex, however it is not clear if this is the case. Instead, we will do a second round of ``helping''. We call a low vertex $v$ a \emph{$k$-donor} if $v$ is in exactly $k$ edges with $1$-flat $2$-vertices which are either type $(M,v,2)$ or are type $(L,v,2)$, where the other low vertex is not helpful. Let $\mathrm{donor}(v)$ denote the value of $k$ for which $v$ is a $k$-donor.
	
	A low vertex $v$ of degree at least $3$ is \emph{half-helpful} if
	\begin{equation}\label{inequality halfhelpful definition}
		d(v)+2\mathrm{supp}(v)+\mathrm{rich}(v)+\frac{1}2\mathrm{reas}(v)\geq \frac{1}2\mathrm{donor}(v)+4.
	\end{equation}
	
	Again, the left-hand side of the above inequality is exactly the charge that $v$ has before it gives any away, and the right-hand side is the amount of charge we want $v$ to give away, along with the charge we want $v$ to keep. Note that since $\mathrm{donor}(v)\leq \mathrm{flat}(v)$, every helpful vertex is also half-helpful. 
	
	We now will show that almost every edge containing a $1$-flat $2$-vertex contains either a helpful or half-helpful vertex. In particular, we will show the following:
	\begin{itemize}
		\item Almost every $(V,L,2)$ edge containing a $1$-flat $2$-vertex contains a vertex of degree at least $4$, and
		\item Almost every vertex of degree at least $4$ is helpful or half-helpful.
	\end{itemize}
	As we have already shown that there are almost no $(M,2,2)$ or $(M,3,2)$ edges, to prove the first bullet point above, we need to only deal with $(L,L,2)$ edges.
	
	\begin{lemma}\label{lemma low edge containing 1flat 2vtx contains 4+vertex}
	Let $e=\{a,b,t\}$ be an $(L,L,2)$ edge with  a $1$-flat $2$-vertex $t$. If $\max\{d(a),d(b)\}\leq 3$, then
	$\{a,b,t\}\subseteq \mathbf{R}'_8$.
	\end{lemma}

	\begin{proof} Suppose $\max\{d(a),d(b)\}\leq 3$ and
	$\{a,b,t\} \not\subseteq \mathbf{R}'_8$.  Then no  vertex of $e$ is in $\mathbf{R}_8$.
		 By Lemma~\ref{lemma no (L,2,2) edges},
		  $d(a)=d(b)=3$. Since $t$ is not in $\mathbf{R}_1$, it is contained in 
	a high edge $\{h,u,t\}$, say $h\in H$. Since $t$ is not in $\mathbf{R}'_2$,  $ht$ is a good pair. To avoid a $C_3^{(3)}$ in $G$, each   $h,t$-link must contain $u$, say, $\{h,u,a\}$ is an edge of $G$. Consider the non-edges $\{h,t,a\}$ and $\{u,t,a\}$. Since these triples intersect every edge containing $t$ in two vertices, there is an $h,a$-link and a  $u,a$-link avoiding the vertex $t$. 
		
		Assume  that there exist distinct vertices $z,z'\in \{h,u\}$ such that there is a  $z,a$-link that has no edge containing both $z$ and $z'$. Thus there exist vertices $w_1$ and $w_2$ such that $\{z,w_1,w_2\}$ and $\{w_1,a,z'\}$ are edges in $G$ (note that $z'$ must be in this second edge to avoid a triangle with $\{z,z',a\}=\{h,u,a\}$). However, then $\{z,w_1,w_2\}$, $\{w_1,a,z'\}$ and $\{z,z',t\}$ form a $C_3^{(3)}$ in $G$,  a contradiction.
		
		Thus, $\{h,u,s\}$ and $\{s,s',a\}$ are edges for some $s,s'$ such  that $\{s,s'\}\cap \{h,u,a,t\}=\emptyset$. We will consider cases based on whether $b\in \{s,s'\}$ or not. 
		
		\textbf{Case 1:} $s=b$. Under this assumption, $s'\not\in\{a,b,u,t,h\}$ and $d(ab)\geq 2$. By Claim~\ref{claim degree d d neighbor}, we must have $d(ab)=2$, and furthermore, $d(at)=d(bt)=1$. Thus by Claim~\ref{claim good pair double neighbor}, $ab$ is a bad pair. Consider the non-edge $\{a,b,h\}$. There must be an $h,a$-link or an $h,b$-link in $G+\{a,b,h\}$. However every edge containing $a$ or $b$ intersects $\{a,b,h\}$ in two vertices, so no such links exist. This proves the case.
		
		\textbf{Case 2:} $s\neq b$. Note that we may have $b=s'$. Consider the non-edge $\{s,a,t\}$. The only edges containing $a$ or $t$ that do not intersect $\{s,a,t\}$ in two vertices are $\{a,u,h\}$ and $\{t,u,h\}$. Thus, $at$  cannot be the pair of core vertices in the $C_3^{(3)}$ in $G+\{s,a,t\}$, so $s$ must, and there has to be an edge in $G$ of the form $\{s,x,y\}$ for some $x\in\{u,h\}$, $y\not\in\{u,h,a,t\}$. If $y\neq s'$, then $\{s,x,y\}$, $\{a,u,h\}$ and $\{s,s',a\}$ form a $C_3^{(3)}$ in $G$, so we must have $y=s'$. Furthermore, if $b=s'$, then $\{s,s',x\}$, $\{h,u,t\}$ and $\{t,a,b\}$ form a $C_3^{(3)}$ in $G$, so $b\neq s'$. Furthermore, note that there is no edge in $G$ containing exactly one vertex from $\{s,s'\}$ and exactly one from $\{h,u\}$ since any such edge could not contain $a$, and thus would create a $C_3^{(3)}$ in $G$ with $\{s,s',a\}$ and $\{a,h,u\}$. Similarly, there is no edge containing $b$ and exactly one vertex from $\{s,s'\}$.
		
		Since $a$ has no double neighbors,  by Claim~\ref{claim good pair double neighbor}, $ss'$ is a bad pair. Consider the non-edge $\{s,s',t\}$. There must be a  $t,q$-link for some $q\in \{s,s'\}$. This link cannot use  edge $\{t,a,b\}$ since no edge involving $a$ can complete the path, and there is no edge in $G$ containing $b$ and exactly one vertex from $\{s,s'\}$. Similarly, the link cannot use the edge $\{t,h,u\}$ since there is no edge containing exactly one vertex from $\{h,u\}$ and one from $\{s,s'\}$. This contradiction completes the proof.
	\end{proof}
	
	Now we focus on showing that almost all low vertices of degree at least $4$ are half-helpful. Vertices of degree $6$ or more and supported vertices are relatively easy to deal with, but first we need a helpful lemma.
	
	\begin{lemma}\label{lemma 1flat 2 vertex implies rich}
		Let $\{h,u,t\}$ be a $(H,L,2)$ edge where $h\in H$ and $t$ is a $1$-flat $2$-vertex. If neither $u$ nor $t$ is  supported
		or belongs to $\mathbf{R}'_8$, then $d(hu)\geq 3$.
	\end{lemma}

	\begin{proof}
		Since $t$ is $1$-flat and not in $\mathbf{R}'_2$, $ht$ is a good pair. We will consider two cases based on if $d(ut)=1$ or not.
		
		\textbf{Case 1:} $d(ut)=1$. For $ht$ to be a good pair and to avoid a triangle with $\{h,u,t\}$, there must be edges $\{h,u,w\}$ and $\{w,z,t\}$ for some $w,z\not\in \{h,u,t\}$. 
		
		Consider the non-edge $\{h,w,t\}$. Since $\{h,w,t\}$ intersects every edge containing $t$ in two vertices, $hw$ must be a good pair, and furthermore, this $h,w$-link must avoid $t$. Furthermore, this link must use $u$ to avoid a $C_3^{(3)}$ in $G$ with $\{h,u,w\}$. If one of the edges of the link is $\{h,u,v\}$ for some $v\not\in \{w,t\}$, then we are done since $d(hu)\geq 3$, so let us assume otherwise. Then the loose path must use edges of the form $\{w,u,v\}$ and $\{v,v_1,h\}$ for some $v,v_1\not\in \{h,w,u,t\}$, but then $\{h,u,t\}$, $\{u,w,v\}$ and $\{v,v_1,h\}$ form a $C_3^{(3)}$ in $G$, a contradiction.
		
		\textbf{Case 2:} $d(ut)\geq 2$, say $\{u,v,t\}$ is an edge of $G$ with $v\neq h$. 
		Since $t$ is  not in $\mathbf{R}'_3$,  the codegree of $hu$ must be at least $2$, say $\{h,u,z\}$ is an edge. 
		
		First, if $z=v$, then the non-edge $\{h,t,v\}$ intersects every edge of $t$ in at least two vertices, so $hv$ is a good pair. Any $h,v$-link must contain $u$ to avoid a $C_3^{(3)}$ in $G$ with $\{h,u,v\}$, but if one of the edges in this link contains both $h$ and $u$, then $d(hu)\geq 3$, satisfying the hypothesis of the lemma, so we may assume that one of the edges in this path is $\{u,v,a\}$ and another is $\{h,a,b\}$ for some $a,b\not\in \{h,u,v,t\}$. However, this creates a $C_3^{(3)}$ in $G$ with edges $\{h,u,t\}$, $\{u,v,a\}$ and $\{h,a,b\}$. Thus, we may assume $z\neq v$.
		
		Now consider the non-edge $\{u,t,z\}$. Since it intersects both edges containing $t$ in two vertices, $uz$ must be a good pair, and there must be a  $u,z$-link that avoids $t$. Furthermore, this link must contain $h$ to avoid a $C_3^{(3)}$ in $G$ with the edge $\{h,u,z\}$. If this link contains an edge containing both $h$ and $u$, this would satisfy the claim of the lemma, so the link consists of edges of the form $\{h,z,a\}$ and $\{u,a,b\}$ for some $a,b\not\in \{h,u,t\}$. But this gives a $C_3^{(3)}$ in $G$ with edges $\{h,u,t\}$, $\{u,a,b\}$ and $\{h,z,a\}$, a contradiction.
	\end{proof}
	
	\begin{lemma}\label{claim supported or high degree implies helpful}
		Let $v\in L\setminus \mathbf{R}'_8$ and $d(v)\geq 3$. If $v$ is supported or $d(v)\geq 6$, then $v$ is helpful.
	\end{lemma}
	
	\begin{proof} If $v$ is supported, then $\mathrm{supp}(v)\geq 1$ and $\mathrm{flat}(v)\leq d(v)-1$. So, since $d(v)\geq 3$,
\[
d(v)+2\mathrm{supp}(v)- \frac{1}2\mathrm{flat}(v)-4\geq \frac{1}2(d(v)+1)-2\geq 0,
\]
	which yields~\eqref{inequality helpful definition}.
	
	Suppose now that some vertex $v\in L\setminus \mathbf{R}'_8$ with $d(v)\geq 6$ is not helpful. By above, $\mathrm{supp}(v)=0.$ Similarly, if $v$ is contained in a reasonable edge or is a recipient in a rich edge, then this edge does not contain a $1$-flat $3^-$-vertex, and hence
$$	d(v)+\mathrm{rich}(v)+\frac{1}2\mathrm{reas}(v)- \frac{1}2\mathrm{flat}(v)-4\geq d(v)+\frac{1}2-\frac{1}2(d(v)-1)-4=\frac{1}2d(v)-3\geq 0,$$		which again yields~\eqref{inequality helpful definition}. Thus
\begin{equation}\label{r-r}
    \mathrm{rich}(v)+\mathrm{reas}(v)=0.
\end{equation}
Since $v\notin  \mathbf{R}'_1$,  there are a vertex $h\in H$ and a vertex $w\in L$	such that
$\{h,v,w\}$ is an edge in $G$. 

{\bf Case 1:} $d(hv)\geq 2$. By Lemma~\ref{coneighbors},
either $\{h,v\}$ is contained in an edge $e_1=\{h,v,w_1\}$ where $w_1$ is neither a $2$-vertex nor a $1$-flat $3$-vertex, or $\mathrm{flat}(v)\leq d(v)-2$, in which case
$$	d(v)- \frac{1}2\mathrm{flat}(v)-4\geq \frac{1}2 d(v)+1-4\geq 0,$$
and hence~\eqref{inequality helpful definition} holds. Thus, assume the former. Since $d(v)>3$,  $e_1$  is either reasonable or rich.
By~\eqref{r-r}, we conclude that $vh$ is a rich pair. Since $d(w_1)\geq 3$,   $\mathrm{flat}(v)\leq d(v)-1$. Thus, if $v$ is not helpful, then
\begin{equation}\label{d6all1}
    \mbox{\em $d(v)=6$ and each edge containing $v$ apart from $e_1$ contains a $1$-flat $2$-vertex.}
\end{equation}
 In particular, since $vh$ is a rich pair, we have edges
$e_2=\{h,v,w_2\}$ and $e_3=\{h,v,w_3\}$ where $w_2$ and $w_3$ are $1$-flat $2$-vertices.

Then the pair $vh$ cannot be good: the first edge of any $h,v$-link must contain at least one of $w_2$ and $w_3$.
Since $v\notin  \mathbf{R}'_2$,
$v$ is adjacent to another vertex in $M$, say $G$ has an edge $e_4=\{h',v,w'\}$ where
$h'\in M$. By~\eqref{d6all1}, $w'$ must be a $1$-flat $2$-vertex, thus $h'\in H$. Then by Lemma~\ref{lemma 1flat 2 vertex implies rich},
$d(uw')\geq 3$, so by Lemma~\ref{coneighbors} some edge containing $\{h',u\}$ does not contain $1$-flat $2$-vertices. This contradicts~\eqref{d6all1}.

{\bf Case 2:} $d(hv)=1$. Since $v\notin  \mathbf{R}'_3$,  $w\notin  \mathbf{R}_3$.
Then  $w$ is not $1$-flat. Thus, $w$ is a $2$-flat $2$-vertex.
Since $w\notin  \mathbf{R}_4$, the co-degree of $hw$ is 1. If
$d(vw)=2$, say $e_2=\{v,w,h_1\}\in E(G)$, then $\mathrm{flat}(v)\leq d(v)-2$
 and we are done. So $d(vw)=1$, and hence $hv$ is a bad pair. 
 
 Since $v\notin  \mathbf{R}'_2$,
  $v$ is adjacent to a vertex $h'\in M$ distinct from $h$, say $e_2=\{v,h',w'\}\in E(G)$. If $w'$ is not a $1$-flat $2$-vertex, then
 $\mathrm{flat}(v)\leq d(v)-2$
 and we are done again. Suppose $w'$ is  a $1$-flat $2$-vertex. Since $v\notin  \mathbf{R}'_2$, $d(h'v)\geq 2$.
 Then by Lemma~\ref{coneighbors}, some edge containing $\{h',v\}$ does not contain $1$-flat $2$-vertices, and we again get
 $\mathrm{flat}(v)\leq d(v)-2$.
\end{proof}

Let $R_9$ be the set of $1$-flat $2$-vertices $u\in L\setminus \mathbf{R}'_8$ that
are double neighbors with an unsupported  $5$-vertex not in $\mathbf{R}'_8$. Let $\mathbf{R}_9=R_9\cup \mathbf{R}_8' $, and  let 
	$\mathbf{R}'_9= \mathbf{R}_9\cup (L\cap N(L_8\cap \mathbf{R}_9)) $.

\begin{lemma}\label{lemma 1flat 2vtx double neighbor with 5vertex}  
 $|R_9|\leq 60\ell^3$. 
\end{lemma}

\begin{proof} Each  $1$-flat $2$-vertex in $ L\setminus \mathbf{R}'_1$ is adjacent to $x$ or $y$.
For $h\in \{x,y\}$, let $\mathcal{F}(h)$ be the set of
edges $\{h,u,t\}$ such that $u \in L\setminus \mathbf{R}'_8$ is an unsupported $5$-vertex, $t\in L\setminus \mathbf{R}'_8$ is a  $1$-flat $2$-vertex, and $d(ut)=2$. By Lemma~\ref{lemma 1flat 2 vertex implies rich}, $d(hu)\geq 3$ for each such $e\in \mathcal{F}(h)$. Choose one such edge $e_1=\{h,u,t\}$. Let $e_2=\{z,u,t\}$ be the other edge containing $t$, and $e_3=\{h,u,v_1\}$ and $e_4=\{h,u,v_2\}$ be other edges containing $\{h,u\}$. We  consider cases based on whether $z\in \{v_1,v_2\}$ or not.

\medskip
{\bf Case 1:} For at least $10\ell^3$ edges $e_1=\{h,u,t\}\in \mathcal{F}(h)$, $z\in \{v_1,v_2\}$, say $z=v_1$. Let $G_1$ be obtained from $G$ by adding the non-edge $e_5=\{h,t,v_1\}$. By definition, $G_1$ has a triangle $T$ formed by $e_5$ and, say $g_1$ and $g_2$. Since both edges in $G$ containing $t$ have two common vertices with $e_5$, the core vertices of $T$ in $e_5$ are $h$ and $v_1$, say $h\in g_1$. 

In order for $e_3,g_1$ and $g_2$ not to form a triangle in $G$, we need $u\in g_1\cup g_2$. If $u\notin g_2$, then $u\in g_1\setminus g_2$, and the edges $g_1,g_2$ and $e_2$ form a triangle in $G$, a contradiction. Thus we may assume $g_2=\{h,u,a\}$. 
Similarly, if $u\notin g_1$, then $u\in g_2\setminus g_1$, and the edges $g_1,g_2$ and $e_1$ form a triangle in $G$, a contradiction again. So we may assume $g_1=e_4$, and in particular,  $v_2\neq a$. 
Thus the edges containing $u$ are $e_1,e_2,e_3,e_4,g_2$.

\medskip
{\bf Case 1.1:} For some  $e_1=\{h,u,t\}\in \mathcal{F}(h)$,  $uh$ is a good pair, say edges $f_1$ and $f_2$ form a $u,h$-link. In order to avoid triangles, $f_1\cup f_2$
must contain $t,v_1$ and $v_2$.
Since  we know all edges containing $t$ or $u$, the
only possibility for this is that $e_6=\{h,v_1,v_2\}\in E(G)$. But then, since $a\neq v_2$, 
the edges $e_6,g_2$ and $e_1$ form a triangle in $G$.

\medskip
{\bf Case 1.2:} $uh$ is a bad pair for at least $10\ell^3$ edges $e_1=\{h,u,t\}\in \mathcal{F}(h)$ with $v_1=z$. We can try to add the edge
$\{u,h,u'\}$ for each two such edges $e_1$ and $e'_1=\{h,u',t'\}$. By the case, the core vertices in each obtained triangle should be always $u$ and $u'$.
Since $u$ is unsupported, all vertices $u,t,v_1,v_2$ are low. So the common neighbor of $u$ with most of $u'$ should be $a$, which then must participate in at least $3\ell^3$ configurations for corresponding $\{h,u,t\}$. 

\medskip
{\bf Case 1.2.1:} This $a$ is adjacent to $h$, say $e_7=\{h,a,v_4\}\in E(G)$. Since each $v_1$ is low, there is a configuration in which $v_4\neq v_1$. For this configuration, $G$ will have a triangle with edges $e_7,g_2$ and $e_1$, a contradiction.

\medskip
{\bf Case 1.2.2:} This $a$ is not adjacent to $h$. Let $G_2$ be obtained from $G$ by adding new edge $e_8=\{h,t,v_2\}$. By definition, $G_1$ has a triangle $T$ formed by $e_8$ and, say $f_1$ and $f_2$. 

 If the core vertices of $T$ were $t$ and $h$, then the edge of $T$ containing $t$, say $f_1$ can be only $e_2$ and hence
 $f_1\cap f_2\subseteq \{u,v_1\}$. Since we know all edges containing $u$ and each of $e_1$ and $g_1$ shares two vertices with
 $e_8$, if $f_1\cap f_2= \{u\}$, then the only candidate for $f_2$ is $e_3$, but it shares two vertices with $e_2$, a contradiction.
 Thus in this case $f_1\cap f_2= \{v_1\}$, say $f_2=\{v_1,h,v_5\}$. Since $a$ is not adjacent to $h$, $v_5\neq a$. Hence $e_1,g_2$ and $f_2$ form 
a triangle, a contradiction.

If the core vertices of $T$ were $t$ and $v_2$, then again the edge  $f_1$  containing $t$ can be only $e_2$ and hence
 $f_1\cap f_2\subseteq \{u,v_1\}$. Again, if $f_1\cap f_2= \{u\}$, then the only candidate for $f_2$ is $e_3$, but it shares two vertices with $e_2$, a contradiction. So, again $f_1\cap f_2= \{v_1\}$, say $f_2=\{v_1,v_2,v_5\}$. Since $h\in e_8$ and
 $u\in e_2$, $v_5\notin \{h,u\}$. Then $f_2,e_2$ and $e_4$ form a triangle in $G$.
 
 The last possibility is that the core vertices of $T$ are $h$ and $v_2$. We may assume that $v_2\in f_2$.
 Because of $e_4$, $u\in f_1\cup f_2$. We know all edges containing $u$ and among these edges only $e_4$ contains
 $\{u,v_2\}$. Hence $u\in f_1-f_2$. In this case, the only candidate for $f_1$ is $e_3$, and so $f_1\cap f_2=\{v_1\}$.
Thus as in the previous paragraph we may assume  $f_2=\{v_1,v_2,v_5\}$. Since $h\in e_8$ , $v_5\neq h$. 
And we know all edges containing $u$, so $v_5\neq u$. Again, $f_2,e_2$ and $e_4$ form a triangle in $G$.
This finishes Case 1.

\medskip
{\bf Case 2:}  For at least $10\ell^3$ edges $e_1=\{h,u,t\}\in \mathcal{F}(h)$, $z\neq v_1$ and $z\in N(h)$. Let an edge containing $z$ and $h$ be $e_9=\{z,h,z'\}$. If $z'\notin \{u,v_1\}$, then
edges $e_9,e_2$ and $e_3$ form a triangle. If $z'=u$, then we have Case 1 with $z'$ in the role of $v_1$. So, $z'=v_1$.

Let $G_3$ be obtained from $G$ by adding  edge $e_{10}=\{u,t,v_1\}$. By definition, $G_3$ has a triangle $T$ formed by $e_{10}$ and, say $g_1$ and $g_2$. Since both edges in $G$ containing $v$ have two common vertices with $e_{10}$,
the core vertices of $T$ in $e_{10}$ are $u$ and $v_1$.  We may assume $u\in g_1$.
Because of $e_3$, $h\in g_1\cup g_2$. Suppose $h\in g_1$, say $g_1=\{u,h,v_6\}$. Since $g_1$ shares only one vertex with $e_{10}$, $v_6\notin \{t,v_1\}$. Since Case 1 is already covered, $v_6\neq z$. Then
edges $e_9,e_2$ and $e_{10}$ form a triangle. This contradiction implies $h\in g_2-g_1$. In this case, 
edges $g_1,g_2$ and $e_1$ form a triangle. This finishes Case 2.

\medskip
{\bf Case 3:} For at least $10\ell^3$ edges $e_1=\{h,u,t\}\in \mathcal{F}(h)$, $z\neq v_1$ and $z\notin N(h)$. As in Case 2, let $G_3$ be obtained from $G$ by adding  edge $e_{10}=\{u,t,v_1\}$. By definition, $G_3$ has a triangle $T$ formed by $e_{10}$ and, say $g_1$ and $g_2$. Again, the core vertices of $T$ in $e_{10}$ are $u$ and $v_1$.  We may assume $u\in g_1$.
Again, $h\in g_1\cup g_2$. If $h\notin g_1$, then edges $g_1,g_2$ and $e_1$ form a triangle. So, we may assume that
$g_1=\{u,h,v_2\}$.

By the case, $G$ has no edges $\{u,t,v_1\}$ and  $\{u,t,v_2\}$. Hence $uh$ is a bad pair,
and we have
 at least $10\ell^3$ such edges  containing $h$.    We can try to add the edge
$\{u,h,u'\}$ for each two such edges. Since pairs $uh$ and $uh'$ are bad, the core vertices in each obtained triangle should be always $u$ and $u'$. So most of these vertices $u$ must have a common neighbor, say $w$ (of high degree).
Since $u$ is unsupported and $t$ is $1$-flat, all vertices $u,v_1,v_2,z$ are low and so $w\notin \{u,z,v_1,v_2\}$.
Let the edge containing $u$ and $w$ be $e_{11}=\{u,w,u_1\}$. Now we know all edges containing $u$ apart from the fact that we do not know whether $u_1$ is one of $v_1,v_2$ or $z$.

If  $w$ is adjacent to $h$, then we simply repeat the argument of Subcase 1.2.1. So below we assume 
 \begin{equation}\label{320}
  w\notin N(h).
 \end{equation}

\medskip
{\bf Case 3.1:}  $u_1=z$. Then let $G_4$ be obtained from $G$ by adding   edge  $e_{12}=\{u,t,w\}$. By definition, $G_4$ has a triangle $T$ formed by $e_{12}$ and, say $f_1$ and $f_2$. Again, the core vertices of $T$ in $e_{12}$ are $u$ and $w$.  We may assume $u\in f_1$. Now $z\in f_1\cup f_2$. Since we know all edges containing $u$ and three of these edges have $2$ common vertices with $e_{12}$, the only candidates for $f_1$ are $e_3$ and $e_4$.
Neither of them contains $z$, so  $z\in f_2-f_1$. The common vertex of $f_1$ and $f_2$ cannot be $h$ because $w\notin N(h)$,
so by symmetry, we may assume that
$f_2=\{u,u_1,v_1\}$. Then edges $e_2,f_2$ and $e_3$ form a triangle.

\medskip
{\bf Case 3.2:}  $u_1\neq z$, but there is an edge $e_{13}$ containing $z$ and $w$, say $e_{13}=\{z,w,z_1\}$.
If $z_1\neq u_1$, then edges $e_{13},e_2$ and $e_{11}$ form a triangle, so suppose $z_1= u_1$. Furthermore, if 
$z_1=v_i$ for some $i\in [2]$, then edges $e_{13},e_2$ and $e_{2+i}$ form a triangle. So below we assume
 \begin{equation}\label{32}
 z_1= u_1\notin \{v_1,v_2\}. 
 \end{equation}
 
 As in Case 3.1, consider $G_4$  obtained from $G$ by adding   edge  $e_{12}=\{u,t,w\}$.  Again, $G_4$ has a triangle $T$ formed by $e_{12}$ and, say $f_1$ and $f_2$. Again, the core vertices of $T$ in $e_{12}$ are $u$ and $w$.  We may assume $u\in f_1$. Now $u_1\in f_1\cup f_2$. Since we know all edges containing $u$ and three of these edges have $2$ common vertices with $e_{11}$, the only candidates for $f_1$ are $e_3$ and $e_4$. By~\eqref{32},
 neither of them contains $u_1$, so  $u_2\in f_2-f_1$. The common vertex of $f_1$ and $f_2$ cannot be $h$ because $w\notin N(h)$,
so by symmetry, we may assume that
$f_2=\{u,u_1,v_1\}$.

We claim that 
 \begin{equation}\label{321}
  u_1\notin N(h).
 \end{equation}
 
 Indeed, suppose there is an edge $e_{14}=\{h,u_1,u_2\}\in E(G)$. Since $w\notin N(h)$, $u_2\neq w$. Since we know all neighbors of $t$ and $u$, we know that $u_2\notin \{u,t\}$. Then edges $e_{14},e_1$ and $e_{10}$ form a triangle.
 This contradiction proves~\eqref{321}.

Now, consider $G_5$  obtained from $G$ by adding   edge  $e_{15}=\{t,w,u_1\}$.  Again, $G_5$ has a triangle $T$ formed by $e_{15}$ and, say $j_1$ and $j_2$. Suppose first that $t$ is a core vertex of $T$ and
 $t\in j_1$. Then $j_1\in \{e_1,e_2\}$. So, the vertex  $q\in j_1\cap j_2$ is one of $u,h$ or $z$. Since the only edge containing $u$ and at least one of $w$ and $u_1$ is $e_{11}$, and it has two common vertices with $e_{15}$, $q\neq u$. By~\eqref{320} 
 and~\eqref{321}, $q\neq h$. So, $q=z$. But $j_2$ must have exactly one common vertex with $\{w,u_1\}$. Then
 edges $j_{2},e_1$ and $e_{11}$ form a triangle. Therefore, the core vertices of $T$ are $w$ and $u_1$.
 
 We may assume $w\in j_1$. Since the codegree of $wu_1$ is at least $3$, $j_1\cup j_2=\{z,u,v_1\}$. But $u$ is not in an edge in which one vertex is in $\{w,u_1\}$ and the other is in $\{z,v_1\}$. This contradiction finishes Case 3.2.

\medskip
{\bf Case 3.3:}  $z\notin  N(w) $. Note that $h$ and $w$ are the only high vertices adjacent to $v$ or $u$.
We have at least $3\ell^3$ such configurations containing $h$. Fix one such configuration.
Among the remaining $3\ell^3-1$ similar configurations with the same vertices $h$ and $w$, find one with vertices $v',u',z',v'_1,v'_2$ such that the distance from $z'$ to $\{v,u\}$ is at least $3$ (we can do it because $z'$ is not adjacent to $h$ or $w$). Consider $G_6$  obtained from $G$ by adding   edge  $e_{16}=\{t,u,z'\}$.  Again, $G_6$ has a triangle $T$ formed by $e_{16}$ and, say $m_1$ and $m_2$. By the choice of $z'$,  it cannot be a  core vertex, so $t$ and $u$ must be. But both edges
containing $t$ share two vertices with $e_{16}$. This contradiction shows that there are less than 
 $10\ell^3$ edges $e_1=\{h,u,t\}\in \mathcal{F}(h)$ with  $z\neq v_1$ and $z\notin N(h)$. 
 
 Together with Cases 1 and 2, this implies that $|\mathcal{F}(h)|\leq 30\ell^3$ for $h\in \{x,y\}$.
This proves the lemma.
\end{proof}

\begin{lemma}\label{2flat2}
Let $h\in M$. Then for every edge $\{h,u,t\}\in E(G)$ such that neither of $t$ and $u$ is  supported, $d(t)=2$, $d(u)\leq 8$ and $d(hu)=2$, we have $\{u,t\}\subseteq \mathbf{R}'_9 $.
\end{lemma}

\begin{proof} Suppose $G$ contains edges $e_1=\{h,u,t\}$ and $e_1=\{h,u,w\}$ such that $u,t\in  L\setminus  \mathbf{R}_9 $,
$d(t)=2$, 
$d(u)\leq 8$, $d(hu)=2$ and
neither of $t$ and $u$ is  supported.

Since  $t\notin  \mathbf{R}_4' $,
\begin{equation}\label{deg1}
    d(ht)=1.
\end{equation}

{\bf Case 1:} There is an edge $e_3=\{t,w,z\}\in E(G)$ containing $\{t,w\}$.  Let $e_4=\{t,h,w\}$. By~\eqref{deg1}, 
$e_4\notin E(G)$. Let $G'=G+e_4$.
Then $G'$ has a triangle  containing $e_4$. Since both $e_1$ and $e_3$ have two common vertices with $e_4$, 
there is an $h,w$-link avoiding $t$ formed by some edges $g_1$ and $g_2$, say $h\in g_1$. In view of $e_2$,
$u\in g_1\cup g_2$, but since $d(hu)=2$, $u\notin g_1$. Thus $u\in g_2-g_1$, and hence $g_1$ and $g_2$ form also 
an $h,u$-link. Then together with $e_1$ they form a triangle in $G$, a contradiction.

{\bf Case 2:} $d(ut)>1.$ Since $d(t)=2$, $d(ut)=2$, and there is an
edge $e_5=\{t,u,u'\}\in E(G)$. By Case 1, $u'\neq w$, and $e_6=\{t,w,u\}$ is not in $E(G)$. 
 Let $G'=G+e_6$.
Then $G'$ has a triangle containing $e_6$. Since both $e_1$ and $e_5$ contain $\{u,t\}$, 
there is a $u,w$-link avoiding $t$, say with edges $g_1$ and $g_2$ where $u\in g_1$. In view of $e_2$,
$h\in g_1\cup g_2$, but since $d(hu)=2$, $h\notin g_1$. Thus $h\in g_2-g_1$, and hence $g_1$ and $g_2$ form also 
a $u,h$-link. Then together with $e_1$ they form a triangle in $G$, a contradiction.

{\bf Case 3:} Cases 1 and 2 do not hold. 
If the pair $uh$ is good, then because of $e_1$, the degree of $ut$ or of $ht$ would be at least $2$. But by~\eqref{deg1} and Case 2, neither holds.
Hence $uh$ is a bad pair.

Suppose now $th$ is good.
 Let $g_1$ and $g_2$ be the two edges of a $t,h$-link in $G$ with $t\in g_1$. In view of $e_1$,
$u\in g_1\cup g_2$, but since Case 2 does not hold, $u\notin g_1$. Thus $ g_2=e_2$, and 
the common vertex of $g_1$ and $g_2$ is $w$. But then  Case 1 holds, a contradiction.

Since both pairs $uh$ and $th$ are bad and $t\notin  \mathbf{R}_5'  $,  by Lemma~\ref{Daras}, there 
is a pair $\{z,z'\}$ of vertices distinct from $h$ such that $e_7=\{t,z,z'\}\in E(G)$ and 
$e_8=\{u,z,z'\}\in E(G)$. By the case,
$w\notin \{z,z'\}$. Consider again $G'=G+e_6$. One of the pairs contained in $e_6$ must be good.

If $ut$ is good, then in view of $e_1$ either $d(ht)\geq 2$ or $d(uh)\geq 3$. The former inequality is Case 2 and the latter is the claim of our lemma. If $uw$ is good and $g_1$ and $g_2$ are the two edges of a $u,w$-link in $G$ with $u\in g_1$, then
in view of $e_2$,
$h\in g_1\cup g_2$, but since $d(hu)=2$, $h\notin g_1$. Thus $h\in g_2\setminus g_1$, and hence $g_1$ and $g_2$ form a $u,h$-link. Then together with $e_1$ they form a triangle in $G$, a contradiction.

The last possibility is that $tw$ is good. Let $g_1$ and $g_2$ be the two edges of a $t,w$-link in $G$ with $t\in g_1$.
Then, since $e_1$ shares two vertices with $e_6$, $g_1=e_7$. By the symmetry between $z$ and $z'$, we may assume that
$g_2=\{w,z,w'\}$. If $w'\neq h$, then $G$ has a triangle with the edges $g_2,e_8$ and $e_1$. Finally, if $w'=h$, then $G$ has a 
triangle with the edges $g_2,e_7$ and $e_1$.
\end{proof}
	
\begin{lemma}\label{lemma degree 4 double neighbors}
	No unsupported  $4$-vertex $u\in L\setminus \mathbf{R}'_9$ is a double neighbor of a $1$-flat  $2$-vertex.
\end{lemma}

\begin{proof} 
Suppose an unsupported  $4$-vertex $u\in L\setminus \mathbf{R}'_9$ is a double neighbor of a $1$-flat  $2$-vertex $t$,
say $\{h,u,t\},\{u,t,z\}\in E(G)$, where $h\in H$ and $z\in L$. 
 By Lemma~\ref{lemma 1flat 2 vertex implies rich}, $d(hu)\geq 3$, and  since $u\notin \mathbf{R}'_4$, $d(hu)=3$. Let $\{h,u,v_1\}$ and $\{h,u,v_2\}$ be the other two edges containing $h$ and $u$,  possibly $z\in \{v_1,v_2\}$. 

Since $u\notin \mathbf{R}'_2$,
 $hu$ is a good pair. Since $\{u,t,z\}$ is the only edge containing $u$ but not $h$, this edge must be in each $u,h$-link. Furthermore, to avoid a $C_3^{(3)}$ in $G$ with any of the edges $\{h,u,t\}$, $\{h,u,v_1\}$ and $\{h,u,v_2\}$, each $u,h$-link must contain all $t$, $v_1$, and $v_2$. This implies that actually $z\in \{v_1,v_2\}$. Assume without loss of generality $z=v_2$. The second edge of this path must be $\{h,z,v_1\}$.

Consider the non-edge $\{h,t,z\}$. Since this non-edge intersects every edge containing $t$ in two vertices, there must be 
an $h,z$-link that avoids $t$. To avoid a $C_3^{(3)}$ in $G$ with the edge $\{h,u,z\}$, this link must contain $u$. The only edge of $G$ that contains $u$, does not contain $t$, and does not contain both $h$ and $z$ is $\{h,u,v_1\}$.  Furthermore, the second edge in this path must be $\{v_1,z,a\}$ for some $a\not\in \{h,u,v_1,z,t\}$, but then the edges $\{h,u,v_1\}$, $\{v_1,z,a\}$ and $\{u,a,t\}$ form a $C_3^{(3)}$ in $G$, a contradiction.
\end{proof}

	\begin{lemma}\label{lemma rich vertices are helpful}
	Let $h\in M$ and $u\in L\setminus \mathbf{R}'_9$. If  $hu$ is a rich pair, then $u$ is helpful.
	\end{lemma}
	
	\begin{proof}
	Since $hu$ is rich, $d(hu)\geq 3$. Let $e_1=\{h,u,v_1\}$, $e_2=\{h,u,v_2\}$ and $e_3=\{h,u,v_3\}$ be edges in $G$. If $u$ is supported or $d(u)\geq 6$, then by Lemma~\ref{claim supported or high degree implies helpful}, $u$ is helpful. So we assume $u$ is unsupported and $d(u)\leq 5$.
Since $u\notin \mathbf{R}'_4$,
 $d(u)\geq 4$, thus $4\leq d(u)\leq 5$.  We will consider cases based on the degree of $u$ and whether $hu$ is a good pair or a bad pair.
	
	\textbf{Case 1:} $hu$ is a good pair. Since $hu$ is a good pair, there is an $h,u$-link, and it must contain the vertices $v_1,v_2$ and $v_3$. By symmetry, we can assume the link contains the edges $e_4=\{h,v_1,v_2\}$ and $e_5=\{v_2,v_3,u\}$. This implies that neither $v_1$ nor $v_2$ is $1$-flat. 
Since $d(uv_3)\geq 2$, by the definition of  $\mathbf{R}'_9$ (when $d(u)=5$) or
Lemma~\ref{lemma degree 4 double neighbors} (when $d(u)=4$, $v_3$ is not a $1$-flat $2$-vertex. It follows that none of $e_1,e_2,e_3,e_5$ contains a $1$-flat vertex and hence $\mathrm{flat}(u)\leq d(u)-4$.
Then
 \eqref{inequality helpful definition} holds.
	
	\textbf{Case 2:} $hu$ is a bad pair and $d(u)= 5$. By Lemma~\ref{coneighbors}, one of the vertices $v_i$ is not a $1$-flat $2$-vertex. Since $u\notin \mathbf{R}'_2$,
	 $u$ is adjacent to some other vertex $h'\in M$, say $\{u,h',z\}\in E(G)$, where $h'\neq h$, but $z$ may be one of the $v_i$'s. Since $d(u)=5$, $1\leq d(h'u)\leq 2$.
	
	\textbf{Case 2.1:}  $d(h'u)=2$. By Lemma~\ref{coneighbors}, either both edges containing $h'$ and $u$ are not $1$-flat $2$-vertices, implying $\mathrm{flat}(u)\leq 2$, or at least one of the edges containing $h'$ and $u$ is either reasonable or rich with recipient $u$, so $\mathrm{flat}(u)\leq 3$ and $\mathrm{rich}(u)+\frac{1}{2}\mathrm{reas}(u)\geq \frac{1}{2}$. In either case, \eqref{inequality helpful definition} holds, so $u$ is helpful.
	
	\textbf{Case 2.2:} $d(h'u)=1$. If $\{h',u,z\}$ is reasonable or rich with recipient $u$, then $\mathrm{flat}(u)\leq 3$, while $\mathrm{rich}(u)+\frac{1}{2}\mathrm{reas}(u)\geq \frac{1}{2}$, so \eqref{inequality helpful definition} holds. Thus $\{h',u,z\}$ must contain either a $2$-vertex or a $1$-flat $3$-vertex. However since $u\notin \mathbf{R}'_3$,
 $z$ cannot be $1$-flat, so $z$ must be a $2$-flat $2$-vertex.  This implies $\mathrm{flat}(u)\leq 3$. 
	
	Suppose first  that $\{h,u,z\}$ is an edge, say  $z=v_1$. In this case, $d(v_1)=2$, so by 
	Lemma~\ref{coneighbors}, $hu$ is contained in at least one edge that does not contain a $1$-flat $2$-vertex and is not $\{h,u,z\}$. This implies that $\mathrm{flat}(u)\leq 2$, so \eqref{inequality helpful definition} holds. Thus $\{h,u,z\}$ is not an edge.
	
	If $d(uz)=2$, then since $z$ is a $2$-flat $2$-vertex, again $\mathrm{flat}(u)\leq 2$, and so \eqref{inequality helpful definition} holds.
Thus  $d(uz)=1$. In this case, since $u\notin \mathbf{R}'_5$,
 $h'u$ and $h'z$ are bad pairs. So by Lemma~\ref{Daras}, there exist vertices $a,b\in V(G)$ such that $\{u,a,b\}$ and $\{z,a,b\}$ are edges. Since $z$ is $2$-flat, one of $a$ or $b$ must be in $M$, say $a\in M$. Furthermore, since $d(ab)\geq 2$ and $u\notin \mathbf{R}'_4$,
  $d(b)\geq 3$.  So $\{u,a,b\}$ is either reasonable or rich with recipient $u$. In either case, $\mathrm{rich}(u)+\frac{1}{2}\mathrm{reas}(u)\geq \frac{1}{2}$ while $\mathrm{flat}(u)\leq 3$ (as neither $\{h',u,z\}$ nor $\{u,a,b\}$ contain a $1$-flat $2$-vertex), so \eqref{inequality helpful definition} holds, i.e.,  $u$ is helpful.
	
\textbf{Case 3:} $hu$ is a bad pair and $d(u)=4$. Since $u\notin \mathbf{R}'_4$,
 $u$ is adjacent to a vertex in $M$ that is not $h$, say $\{u,h',z\}$ is an edge, where $h'\in M\setminus\{h\}$. Since $d(h'u)=1$
 and $u\notin \mathbf{R}'_3$, $z$ cannot be $1$-flat. 

\textbf{Case 3.1:} $z$ is a $2$-flat $2$-vertex. We claim that
\begin{equation}\label{badpairs}
    \mbox{\em $h'z$ and $uz$ are bad pairs.}
\end{equation}
 Indeed, since $z\in N(u)$ and $u\notin \mathbf{R}'_4$, $d(h'z)=1$. Hence 
 if $d(uz)=1$,  then by Claim~\ref{claim good pair double neighbor}, $h'u$ and $h'z$ are bad pairs.   
 
 Suppose now  $d(uz)=2$, say $z=v_1$.  Then  $h$ and $h'$ are not adjacent since any edge containing $h$ and $h'$ cannot contain $u$ or $z$ and misses one of $v_2$ or $v_3$, so such an edge would create a $C_3^{(3)}$ in $G$ with $\{h',u,z\}$ and either $\{h,u,v_2\}$ or $\{h,u,v_3\}$. 
 If there was an $h',z$-link, it would have to use the edge $\{z,u,h\}$, but there is no edge connecting $u$ or $h$ to $h'$ that does not contain $z$. Thus $h'z$ is bad. If there was an $h',u$-link, then it would  contain $z$ to avoid a $C_3^{(3)}$ in $G$ with the edge $\{h',u,z\}$, and so the edge $\{h,u,z\}$ must be contained in the path, but  since $h$ and $h'$ are not adjacent, there is no edge that can connect $h'$ to one of $h$ or $z$ that does not contain $u$. Thus, $h'u$ is also bad. This proves~\eqref{badpairs}.

Since $u\notin \mathbf{R}'_5$, by~\eqref{badpairs},
 there exist edges $\{u,a,b\}$ and $\{z,a,b\}$ for some $a,b\in V(G)$. As we already know all edges containing $u$, $\{u,a,b\}=\{h,u,v_i\}$ for some $i\in \{1,2,3\}$, assume without loss of generality $\{u,a,b\}=\{h,u,v_1\}$. Then consider the non-edge $\{h,u,z\}$, and note that this intersects both edges containing $z$ in two vertices, so 
there should be an $h,u$-link.  But $hu$ is a bad pair, a contradiction.

\textbf{Case 3.2:} $z$ is not a $2$-flat $2$-vertex. Then $\{h',u,z\}$ is either reasonable or rich with recipient $u$. By Lemma~\ref{coneighbors}, one of the vertices $v_i$ is not a $1$-flat $2$-vertex, so $\mathrm{flat}(u)\leq 2$, and $\mathrm{rich}(u)+\frac{1}{2}\mathrm{reas}(u)\geq \frac{1}{2}$. Thus, \eqref{inequality helpful definition} holds unless  $\mathrm{flat}(u)=2$, $\mathrm{rich}(u)=0$, and $\mathrm{reas}(u)=1$. By symmetry, assume that $v_1$ is not a $1$-flat $2$-vertex, while $v_2$ and $v_3$ are. By Lemma~\ref{lemma degree 4 double neighbors}, $d(uv_2)=d(uv_3)=1$. Furthermore, 
Since $u\notin \mathbf{R}'_2$,
 $hv_2$ an $hv_3$ are good pairs. Consider first an $h,v_2$-link. To avoid a $C_3^{(3)}$ in $G$ with $\{h,u,v_2\}$, $u$ must be in this link. Since $d(uv_2)=1$, either $\{h,u,v_1\}$ or $\{h,u,v_3\}$ must be in this link, but if $\{h,u,v_3\}$ was in this link, then the second edge would need to contain $v_2$ and $v_3$, which contradicts Lemma~\ref{lemma no (L,2,2) edges}. Thus, the  $h,v_2$-link must contain the edge $\{h,u,v_1\}$ and an edge $\{v_1,v_2,w_1\}$. Similarly, there is an $h,v_3$-link containing the edges $\{h,u,v_1\}$ and $\{v_1,v_3,w_2\}$ for some vertices $w_1$ and $w_2$, possibly equal. This gives us all the edges incident with $v_2$ or $v_3$.

Consider the non-edge $\{v_1,v_2,v_3\}$, and let $T$ be a $C_3^{(3)}$ in $G+\{v_1,v_2,v_3\}$. Note that $v_2$ and $v_3$ cannot both be core vertices in $T$ since the only edges containing $v_2$ or $v_3$ that do not contain two vertices in $\{v_1,v_2,v_3\}$ are $\{h,u,v_2\}$ and $\{h,u,v_3\}$, which share two vertices. Thus, $v_1$ is a core vertex in $T$, and we may assume by symmetry that $v_2$ is as well. Since $v_2$ is a core vertex, $\{h,u,v_2\}$ must be an edge of $T$, which implies that $v_1$ is in an edge containing one of $h$ or $u$, but not both. If $v_1$ is in an edge containing $u$ but not $h$, this edge must be $\{h',u,z\}$ (i.e. $v_1=z$), but then $\{h,u,v_2\}$, $\{v_1,v_2,w_1\}$ and $\{v_1,u,h'\}$ form a $C_3^{(3)}$ in $G$, a contradiction. So there must be an edge containing $v_1$ and $h$, but not $u$. This implies that $d(hv_1)\geq 2$. If $d(hv_1)\geq 3$, then $hv_1$ is a rich pair so $\{h,u,v_1\}$ is reasonable, which makes $\mathrm{reas}(u)=2$, giving us that $u$ helpful. Similarly, if $d(hv_1)=2$, then $\{h,u,v_1\}$ is exceptional and thus reasonable, so again $u$ is helpful.
\end{proof}

\begin{claim}\label{claim exceptional vertices are helpful}
     Let $\{h,u,v\}$ be an exceptional edge such that $h\in M$, and $v\in L\setminus \mathbf{R}'_9$ is the exception in this edge. Then $v$ is helpful.
\end{claim}

\begin{proof}
    Let $\{h,v,z\}$ be the second edge in $G$ that contains $h$ and $v$, and let $\{h,u,w_1\}$ and $\{h,u,w_2\}$ be the remaining edges of $G$ that contain the rich pair $hu$, and note that $w_1$ and $w_2$ must be $1$-flat $2$-vertices. By Lemma~\ref{lemma degree 4 double neighbors}, $d(uw_1)=d(uw_2)=1$. If at least one of $hw_1$ and $hw_2$ is a bad pair, then some of $w_1$ and $w_2$ is in
    $\mathbf{R}'_2$. In this case, $u\in \mathbf{R}'_3$ and hence $v\in \mathbf{R}'_4$, a contradiction. Thus $hw_1$ and $hw_2$
    are good pairs. Since Lemma~\ref{lemma no (L,2,2) edges} implies that $w_1$ and $w_2$ are not adjacent, and any  $h,w_1$-link must contain $u$ to avoid a $C_3^{(3)}$ in $G$ with the edge $\{h,u,w_1\}$, any  $h,w_1$-link must use the edge $\{h,u,v\}$ and an edge $\{v,w_1,z_1\}$. Similarly, 
    any $h,w_2$-link must use the edge $\{h,u,v\}$ and an edge $\{v,w_2,z_2\}$. We claim that $z=z_1=z_2$. If $z\neq z_1$, then $\{h,v,z\}$, $\{h,u,w_1\}$ and $\{v,w_1,z_1\}$ for a $C_3^{(3)}$ in $G$, a contradiction. Similarly if $z\neq z_2$, then $\{h,v,z\}$, $\{h,u,w_2\}$ and $\{v,w_2,z_2\}$ form a $C_3^{(3)}$ in $G$, again a contradiction. Thus, $z=z_1=z_2$. This implies that $d(z)\geq 3$. We will consider cases based on $d(z)$.
    
    \textbf{Case 1:} $d(z)\geq 4$. Then $\{h,v,z\}$ is either reasonable or rich with recipient $v$. Along with the reasonable edge $\{h,u,v\}$, this gives us that $\frac{1}{2}\mathrm{reas}(v)+\mathrm{rich}(v)\geq 1$, while $\mathrm{flat}(v)\leq d(v)-2$, which along with $d(v)\geq 4$ is sufficient to satisfy \eqref{inequality helpful definition}, so $v$ is helpful.
    
    \textbf{Case 2:} $d(z)=3$. Then $d(vz)=d(z)$, so by Lemma~\ref{lemma few i vertices that are i-neighbors of h}, $d(z)\geq 5$. Since $\{h,v,z\}$ does not contain a $1$-flat $2$-vertex and $\{h,u,z\}$ is reasonable,  $\mathrm{flat}(v)\leq d(v)-2$, while $\mathrm{reas}(v)\geq 1$ and $d(v)\geq 5$, which implies \eqref{inequality helpful definition}, so $v$ is helpful.
\end{proof}
	
\begin{lemma}\label{lemma 2flat 5vertices are helpful}
    All   $2^+$-flat $5$-vertices outside of $\mathbf{R}'_9$ are helpful.
\end{lemma}
	
\begin{proof}  Suppose there is   a   $2^+$-flat $5$-vertex $u\notin \mathbf{R}'_9$ that is not helpful. By Lemma~\ref{claim supported or high degree implies helpful}, $u$ is unsupported. 
By Lemma~\ref{lemma rich vertices are helpful} and Claim~\ref{claim exceptional vertices are helpful}, we can assume $u$ is not rich or exceptional. Furthermore, by Lemma~\ref{lemma 1flat 2 vertex implies rich}, any non-low edge containing $u$ must not contain a $1$-flat $2$-vertex, since this would imply $u$ is rich. Thus, if $u$ is $3^+$-flat, then $\mathrm{flat}(u)\leq 2$, and \eqref{inequality helpful definition} is satisfied. Thus, we may assume $u$ is $2$-flat. Furthermore, since $\mathrm{flat}(u)\leq 3$, if $u$ is in at least one reasonable edge or a rich edge with recipient $u$, then \eqref{inequality helpful definition} is satisfied, so we must have $\mathrm{reas}(u)=\mathrm{rich}(u)=0$. We will consider cases based on if $u$ has one or two neighbors in $M$.
		
{\bf Case 1:} There exists a vertex $h\in M$ such that $d(uh)=2$. Let $\{h,u,v_1\}$ and $\{h,u,v_2\}$ be edges of $G$. Note that neither $v_1$ nor $v_2$ are $1$-flat $2$-vertices, and furthermore by Lemma~\ref{2flat2}, $v_1$ and $v_2$ are not $2$-flat $2$-vertices. Then since $\mathrm{reas}(u)=\mathrm{rich}(u)=0$, $v_1$ and $v_2$ must be $1$-flat $3$-vertices. Note that if $\{u,v_1,v_2\}$ is an edge, then $\mathrm{flat}(u)\leq 2$, and hence $u$ is helpful, so $\{u,v_1,v_2\}\not\in E(G)$. 

 Since $u$ is $2$-flat, by the case, $h$ is the only non-low neighbor of $u$, so by Lemma~\ref{lemma few bad pairs with low degrees in G-h} $hu$ must be a good pair. Any $h,u$-link $L$ must contain both $v_1$ and $v_2$, but $\{v_1,v_2,u\}$ is not an edge of $G$, so the edge in $L$
 that contains $h$ must contain  $v_1$ or $v_2$, contradicting the fact that they both are $1$-flat.

\textbf{Case 2:} There exist distinct vertices $h,h'\in M$ such that $d(hu)=d(h'u)=1$. Let $\{h,u,v_1\}$ and $\{h',u,v_2\}$ be edges in $G$, noting that we may have $v_1=v_2$. Since $u\notin \mathbf{R}'_3$,
 $v_1$ cannot be $1$-flat, and so $v_1$ is a $2$-flat $2$-vertex. Similarly, $v_2$ is a $2$-flat $2$-vertex.

\textbf{Case 2.1:} $v_1=v_2$. Let $\{u,t_1,w_1\}$, $\{u,t_2,w_2\}$ and $\{u,t_3,w_3\}$ be the low edges containing $u$ with $t_1$, $t_2$ and $t_3$ being $1$-flat $2$-vertices. Note that while all $t_i$ are distinct, the $w_i$'s do not need be. First consider if $d(uw_i)=1$ for some $i\in\{1,2,3\}$. In this case, consider the non-edge $\{v_1,u,t_i\}$. It intersects all edges containing $v_1$ in two vertices, so 
there is a $u,t_i$-link . Let $\{u,t_{i'},w_{i'}\}$ be the edge of this link containing $u$, and note that the second edge must be of the form $\{t_i,w_{i'},z\}$ since $t_i$ and $t_{i'}$ are not adjacent by Lemma~\ref{lemma no (L,2,2) edges} and
the fact that $t_i$ and $t_{i'}$ are not in $\mathbf{R}'_5$.
However, in this case, $\{u,t_i,w_i\}$, $\{u,t_{i'},w_{i'}\}$ and $\{t_i,w_{i'},z\}$ form a $C_3^{(3)}$ in $G$, where all these vertices are distinct because $d(uw_i)=1$, and $z$ must actually be a vertex in $M$. Thus, we cannot have $d(uw_i)=1$ for any $i\in\{1,2,3\}$. This implies that  $w_1=w_2=w_3$. In this case,  consider the non-edge $\{v_1,u,w_1\}$. It intersects every edge containing  $v_1$ or $u$ in at least two vertices, so adding this non-edge cannot create a $C_3^{(3)}$ in $G$, a contradiction.


\textbf{Case 2.2:} $v_1\neq v_2$. In this case, $d(uv_1)=d(hv_1)=1$, so by Claim~\ref{claim good pair double neighbor}, $hu$ and $hv_1$ are bad pairs. Then since $u\notin \mathbf{R}'_5$,
 there  exist vertices $a,b\in V(G)$ such that $\{u,a,b\}$ and $\{v_1,a,b\}$ are edges in $G$. 
 Since $u$ is $2$-flat and  $v_1$ is a $2$-flat $2$-vertex, $\{u,a,b\}$ must be the edge $\{h',u,v_2\}$. But then  $d(h'v_2)=2=d(v_2)$, which contradicts the fact that $u\notin \mathbf{R}'_4$.
\end{proof}

\begin{lemma}\label{lemma 2flat 4vertices are helpful}
    All   $2^+$-flat $4$-vertices outside of $\mathbf{R}'_9$ are helpful.
\end{lemma}

\begin{proof}
    Suppose there is   a   $2^+$-flat $4$-vertex $u\notin \mathbf{R}'_9$ that is not helpful. By Lemma~\ref{claim supported or high degree implies helpful}, $u$ is unsupported. 
    By Lemma~\ref{lemma rich vertices are helpful} and Claim~\ref{claim exceptional vertices are helpful}, we are done if $u$ is rich or exceptional, so we may assume $u$ is neither. Let $h\in M$ be a neighbor of $u$.
    
		\textbf{Case 1:} $u$ is $4$-flat. In this case, $u$ is only in high edges, so if $u$ has a $1$-flat  $2$-neighbor, $u$ is rich by Lemma~\ref{lemma 1flat 2 vertex implies rich},  contradicting our earlier assumption. Thus $u$ has no $1$-flat $2$-neighbors, and is trivially helpful.
		
		\textbf{Case 2:} $u$ is $3$-flat. Since $u$ is not special, it has a neighbor $h\in H$. If $h$ is the only non-low neighbor of $u$, then $hu$ is a rich pair, and we are done by Lemma~\ref{lemma rich vertices are helpful}.
		 Thus assume $u$ is adjacent to at least two vertices in $M$, and consequently $d(h'u)=1$ for some $h'\in M$ (possibly $h=h'$), say $\{h',u,z\}$ is an edge for some $z\in L$. Furthermore, none of the non-low edges containing $u$ can contain a $1$-flat $2$-vertex, so  only the single low edge containing $u$ may contain a $1$-flat $2$-vertex. Thus, since $u$ is not helpful,
		  $\mathrm{flat}(u)=1$ and $\mathrm{rich}(u)=\mathrm{reas}(u)=0$. This implies that $z$ is either a $2$-vertex or a $1$-flat $3$ vertex, but since $u\notin \mathbf{R}'_3$,  $z$ is not $1$-flat. So $z$ is a $2$-flat $2$-vertex. By Claim~\ref{2flat2} $d(uz)=1$, and by Claim~\ref{claim degree d d neighbor} $d(h'z)=1$, so by Claim~\ref{claim good pair double neighbor}, $h'u$ and $h'z$ are bad pairs. Thus since $u\notin \mathbf{R}'_5$, there exist vertices $a$ and $b$ such that $\{u,a,b\}$ and $\{z,a,b\}$ are edges. Since $z$ is $2$-flat,  one of $a$ and  $b$ is in $M$, say $a\in M$. Then $d(b)\geq 3$ by Claim~\ref{claim degree d d neighbor} since $d(ab)\geq 2$, and thus  edge $\{u,a,b\}$ is either reasonable or rich with recipient $u$. In  both cases we get a contradiction with the fact that $\mathrm{rich}(u)=\mathrm{reas}(u)=0$.

		\textbf{Case 3:} $u$ is $2$-flat and $d(hu)=2$. Let $\{h,u,v_1\}$ and $\{h,u,v_2\}$ be the edges of $G$ containing $h$ and $u$.
		
		\textbf{Case 3.1:} $\{h,v_1,v_2\}\in E(G)$. Then neither $v_1$ nor $v_2$ is $1$-flat. Moreover, 
		since $u\notin \mathbf{R}'_4$,
		 $d(v_1),d(v_2)\geq 3$. Thus, the edges $\{h,u,v_1\}$ and $\{h,u,v_2\}$ are either reasonable or rich with recipient $u$, so  we  have $\mathrm{rich}(u)+\frac{1}2\mathrm{reas}(u)\geq \frac{1}2\mathrm{flat}(u)$. Then \eqref{inequality helpful definition} holds and $u$ is helpful. Thus, below we  assume $\{h,v_1,v_2\}\not\in E(G)$. 
		
		
		\textbf{Case 3.2:}  $\{u,v_1,v_2\}\in E(G)$. Since $u$ is unsupported,  by Lemma~\ref{lemma degree 4 double neighbors} neither $v_1$ nor $v_2$  has degree $2$. 
		
		Suppose both $v_1$ and $v_2$ are  $1$-flat $3$-vertices. Since $u\notin \mathbf{R}'_4$,
	  $d(uv_1)< d(v_1)=3$ and $d(uv_2)< d(v_2)=3$. Hence $d(uv_1)=d(uv_2)= 2$.
		 Since $u\notin \mathbf{R}'_2$,
		 $hv_1$ is a good pair, and any  
		 $h,v_1$-link needs to contain $u$ to avoid a $C_3^{(3)}$ in $G$. The only edges containing $u$ that could be in this link are $\{h,u,v_2\}$ and $\{u,v_1,v_2\}$. Edge $\{u,v_1,v_2\}$ cannot be there since neither  $u$ nor $v_2$ is in an edge with $h$ that contains only one of these vertices, and if $\{h,u,v_2\}$ is in this link, then the other edge of the link must be $\{v_1,v_2,z\}$ for some $z\not\in \{h,u,v_1,v_2\}$. But then the non-edge $\{h,v_1,v_2\}$ intersects all edges containing either $v_1$ or $v_2$ in at least two vertices, so neither can be a core vertex in the $C_3^{(3)}$ in $G+\{h,v_1,v_2\}$, a contradiction. Thus, at least one of $v_1$ or $v_2$ is not a $1$-flat $3$-vertex, say $v_1$ is not. 
		 
		 Then $\{h,v_1,u\}$ is either reasonable or a rich edge with recipient $u$, and $\mathrm{flat}(u)\leq 1$ since $\{h,v_1,u\}$, $\{h,v_2,u\}$ and $\{u,v_1,v_2\}$ all do not contain a $1$-flat $2$-vertex. Thus $u$ is helpful since $\mathrm{rich}(u)+\frac{1}{2}\mathrm{reas}(u)\geq \frac{1}{2}\mathrm{flat}(u)$, i.e., \eqref{inequality helpful definition} is satisfied.
		
		\textbf{Case 3.2:} $\{u,v_1,v_2\}\not\in E(G)$. Since $u\notin \mathbf{R}'_2$,
		 $hu$ is a good pair. In particular, there is an $h,u$-link, say with edges $e_1$ and $e_2$, where $h\in e_1$ and $u\in e_2$. To avoid a $C_3^{(3)}$ in $G$ with $\{h,u,v_1\}$ and $\{h,u,v_2\}$, $v_1$ and $v_2$ must both be in this link. Since $\{h,v_1,v_2\}, \{u,v_1,v_2\}\not\in E(G)$, we can assume without loss of generality that $e_1=\{h,v_1,z\}$ and $e_2=\{z,v_2,u\}$ for some $z\in V(G)\setminus\{h,u,v_1,v_2\}$. 
		
		If $z$ is a $1$-flat $2$-vertex, consider the non-edge $\{h,u,z\}$. Since this non-edge intersects all edges containing $z$ in two vertices, there is an $h,u$-link  that avoids $z$. Again, this link must contain $v_1$ and $v_2$ to avoid a $C_3^{(3)}$ in $G$. If this link has edges $\{h,v_1,z'\}$ and $\{z',v_2,u\}$ for $z'\neq z$, then $d(hv_1)\geq 3$ so $v_1$ is not a $2$-vertex or $1$-flat $3$-vertex, and $d(uv_2)\geq 3$. So by Claim~\ref{claim degree d d neighbor}, $d(v_2)\geq 4$ since $d(u)=4$, and thus the edges $\{h,u,v_1\}$ and $\{h,u,v_2\}$ are either rich with recipient $u$ or reasonable. The other possible $h,u$-link  that avoids $z$ is a path with edges $\{h,v_2,z'\}$ and $\{z',u,v_1\}$ for some $z'\neq z$. In this case, $d(hv_1)\geq 2$ and $d(hv_2)\geq 2$, so since $u\notin \mathbf{R}'_4$,
		 $v_1$ and $v_2$ are not $2$-vertices. Since they also are not $1$-flat, $\{h,v_1,u\}$ and $\{h,v_2,u\}$ are either rich with recipient $u$ or reasonable. Furthermore, in either of these two cases, $\mathrm{flat}(u)\leq 2$, so $u$ is helpful since $\mathrm{rich}(u)+\frac{1}{2}\mathrm{reas}(u)\geq \frac{1}{2}\mathrm{flat}(u)$ implies \eqref{inequality helpful definition}. Thus, we are done if $z$ is a $1$-flat $2$-vertex. 
		
		If $z$ is not a $1$-flat $2$-vertex, then $\mathrm{flat}(u)\leq 1$, and since $d(hv_1)\geq 2$ and  $u\notin \mathbf{R}'_4$,
		 $v_1$ is not a $2$-vertex and  not $1$-flat. So $\{h,v_1,u\}$ is either rich with recipient $u$ or reasonable. Therefore,  again $u$ is helpful since $\mathrm{rich}(u)+\frac{1}{2}\mathrm{reas}(u)\geq \frac{1}{2}\mathrm{flat}(u)$.
		
		\textbf{Case 4:} $u$ is $2$-flat and not a double neighbor 
		with any vertices in $M$. Let $h,h'\in M$ be neighbors of $u$ and let  $\{h,u,z\}$ and $\{h',u,z'\}$ be edges in $G$. If either of these edges are rich with recipient $u$ or if both of them are reasonable, then $\mathrm{rich}(u)+\frac{1}{2}\mathrm{reas}(u)\geq \frac{1}{2}\mathrm{flat}(u)$, and so $u$ is helpful. Thus we may assume at least one of $z$ or $z'$ is either a $2$-vertex or a $1$-flat $3$-vertex, assume $z$ is such a vertex. Since $d(hu)=1$
		and since $u\notin \mathbf{R}'_3$,
	 $z$ is not $1$-flat, and thus must be a $2$-flat $2$-vertex.
		
		\textbf{Case 4.1:} $d(uz)=2$. For $u$ to not be helpful, we must have $\mathrm{flat}(u)\geq 1$. Let $\{u,t,w\}$ be a low edge containing $u$ with $t$ a $1$-flat $2$-vertex, and let $\{u,a_1,a_2\}$ be the other low edge containing $u$. Note that $t,w\not\in \{h,h',z\}$, and by Lemma~\ref{lemma degree 4 double neighbors}, $t\not\in \{a_1,a_2\}$. Consider the non-edge $\{u,t,z\}$. Since this non-edge intersects every edge containing $z$ in two vertices, 
		there a $u,t$-link. The only edge containing $u$ that can be in this link is $\{u,a_1,a_2\}$. So the other edge must connect $a_1$ or $a_2$ to $t$, say this edge is $\{t,a_1,h''\}$ (note that $h''\in M$ since $t$ is adjacent to some vertex in $M$), but then we may assume $w\in \{a_1,a_2\}$ since otherwise $\{u,t,w\}$, $\{u,a_1,a_2\}$ and $\{t,a_1,h''\}$ form a $C_3^{(3)}$ in $G$. Then consider the non-edge $\{u,t,w\}$, and note that this non-edge intersects every edge containing either $u$ or $t$ in two vertices, so adding $\{u,t,w\}$ to $G$ does not create a $C_3^{(3)}$, a contradiction.

		\textbf{Case 4.2:} $d(uz)=1$. Since $u\notin \mathbf{R}'_4$, $d(zh)=1$, so by Claim~\ref{claim good pair double neighbor}, $hu$ and $hz$ are bad pairs, and thus since $u\notin \mathbf{R}'_5$,
		 there exist vertices $a,b\in V(G)$ such that $\{u,a,b\}$ and $\{z,a,b\}$ are edges of $G$. Since $u$ is $2$-flat and $z$ is a $2$-flat $2$-vertex, $\{u,a,b\}$ must be the edge $\{h',u,z'\}$. This implies that $d(h'z)\geq 2$. If $d(h'z)\geq 3$ or if $d(z)\geq 8$, then $\mathrm{rich}(u)\geq 1$, and $\mathrm{flat}(u)\leq 2$, so \eqref{inequality helpful definition} holds. On the other hand, if $d(h'z)=2$ and $d(z)\leq 7$, then this contradicts Lemma~\ref{2flat2}.
\end{proof}

Let $R_{10}$ be the set of $1$-flat $2$-vertices $t$ such that the low edge $\{u,v,t\}$ containing $t$
satisfies the following:
\begin{enumerate}[label=(\alph*)]
\item none of $u,v$ or $t$ is in $\mathbf{R}'_9$;
\item for some $4\leq k\leq 5$, $u$ is a $1$-flat unsupported $k$-vertex;
\item $v$ is not helpful;
\item the high vertex $h_1$ adjacent to $t$ is distinct from the high vertex $h_2$ adjacent to $u$.
\end{enumerate}

Let $\mathbf{R}_{10}=R_{10}\cup \mathbf{R}_9' $, and  let 
	$\mathbf{R}= \mathbf{R}_{10}\cup (L\cap N(L_8\cap \mathbf{R}_{10})) $.

\begin{lemma}\label{lemma 1flat vertices adjacent to 1flat 2 vertices}
$|R_{10}|\leq n/\ell$.
\end{lemma}

\begin{proof} Let $\mathcal{F}$ be the set of low edges $\{u,v,t\}$ satisfying (a)--(d) for some $1$-flat $2$-vertex $t$. Let $\mathcal{F}_1$ be the set of edges in $\mathcal{F}$ where $h_1=x$ and $h_2=y$. If the lemma does not hold, then by the symmetry between $x$ and $y$, we may assume that  $|\mathcal{F}_1|\geq \frac{n}{2\ell}$. For $e_1=\{u,v,t\}\in\mathcal{F}_1$, let $e_2=\{x,t,v_1\}$ be the unique high edge containing $t$ and
$e_3=\{y,u,u_1\}$ be the unique high edge containing $u$. By Lemma~\ref{lemma 1flat 2 vertex implies rich},
\begin{equation}\label{rich-v1}
    \mbox{\em $v_1$ is rich or supported.}
\end{equation}

So if $v_1=v$, then $v$ is rich or supported, and thus helpful by either Lemma~\ref{lemma rich vertices are helpful} or Lemma~\ref{claim supported or high degree implies helpful}, contradicting (c). So, $v_1\neq v$. Since $x\notin N(u)$, $v_1\neq u$. If $v_1=u_1$, then since $v_1\neq v$, the edges $e_1,e_2$ and $e_3$ form a triangle, a contradiction. Thus, $v_1\notin \{v,u,u_1\}$.

Since $t\notin \mathbf{R}'_2$, $xt$ is a good pair. Let edges $g_1$ and $g_2$ form a $x,t$-link. Since $d(t)=2$, $g_1=e_1$. Since $x\notin N(u)$, $g_2$ has the form $\{v,x,w\}$ for some vertex $w$. If $w\neq v_1$, then  $e_1,e_2$ and $g_2$ form a triangle in $G$. Hence $w=v_1$. Since each of $e_2$ and $g_2$ contains only one vertex in $M$,~\eqref{rich-v1} yields that $d(v_1)\geq 3$.
 
If $v=u_1$, then $v$ is a $2^+$-flat vertex with $\mathrm{flat}(v)\leq d(v)-2$ (because of $e_3$ and $g_2$) and $\mathrm{rich}(v)\geq 1$ (by~\eqref{rich-v1}). Moreover, if $d(v)=3$, then $e_3$ is reasonable. This contradicts the condition that $v$ is not helpful. Thus $v\neq u_1$.

Suppose  pair $vt$ is good and edges $f_1$ and $f_2$ form a $v,t$-link. Since $d(t)=2$, $f_1=\{x,t,v_1\}$. In order for the edges $f_1,f_2$ and $e_1$ not to form a triangle, $u\in f_1\cup f_2$. Since $u\notin N(x)$,
$x\in f_2\setminus f_1$, which yields $f_2=\{x,v,v_1\}$. Now, consider adding the non-edge $e_5=\{u,t,v_1\}$. Since $e_5$ shares two vertices with each of the two edges containing $t$, $G$ must have edges $f_3$ and $f_4$ forming a $u,v_1$-link. In view of $f_2$, $v\in f_3\cup f_4$. Since $u\notin N(x)$, $f_4\neq g_2$. It follows that $d(v)\geq 4$ (because $v$ is contained in $e_1,f_2,g_2$ and at least one of $f_3$ and $f_4$), $\mathrm{flat}(v)\leq d(v)-2$ (because of $f_2$ and $g_2$), and $\mathrm{rich}(v)\geq 1$  (because of  $g_2$). So,
\begin{equation}\label{v help}
d(v)+\mathrm{rich}(v)-\frac{1}{2}\mathrm{flat}(v)-4\geq \frac{1}{2}d(v)+1+1-4\geq 0,
\end{equation}
contradicting (c). Thus for each $e_1=\{u,v,t\}\in\mathcal{F}_1$, the pair $vt$ is bad.

Since $u\notin \mathbf{R}'_4$, either $d(u_1)>3$ or $d(u_1)>d(yu_1)$.

If $G$ has an edge $e_6=\{y,v,w_2\}$ containing $y$ and $v$, then in order not to have triangle with edges $e_6,e_1$ and $e_3$, $w_2\in \{u,t,u_1\}$. On the other hand, we know both edges of $G$ containing $t$, so $w_2\neq t$, and we know that unique edge containing $\{y,u\}$ does not contain $w$. Thus, $w_2=u_1$. Since $d(yu_1)\geq 2$, $d(u_1)\geq 3$. Therefore, $\mathrm{flat}(v)\leq d(v)-2$ (because of $e_6$ and $g_2$), and $\mathrm{rich}(v)\geq 1$  (because of  $g_2$). So, if $d(v)\geq 4$, then as in~\eqref{v help}, $v$ is helpful, contradicting (c). Moreover, if $d(v)=3$, then $v$ is not rich, and hence either $e_6$ is reasonable or $u_1$ is rich, so
\[
d(v)+\mathrm{rich}(v)+\frac{1}{2}\mathrm{reas}(v) -\frac{1}{2}\mathrm{flat}(v)-4\geq \frac{1}{2}d(v)+1+1+\frac{1}{2}-4= 0.
\]
It follows that $y\notin N(v)$.

Since $u$ is $1$-flat, at most $9\ell$ vertices not adjacent to $y$ are at distance at most $2$ from $u$. Recall that for each $e'_1=\{u',v',t'\}\in\mathcal{F}_1$, $v'$ and $t'$ are not adjacent to $y$. So, since $|\mathcal{F}_1|\geq \frac{n}{2\ell}$, there is an edge $e'_1=\{u',v',t'\}\in\mathcal{F}_1$ such that the distance from $u$ to $\{v',t'\}$ is at least $3$. Then, since $v't'$ is a bad pair, adding the non-edge $\{v',t',u\}$ to $G$ does not create a triangle, a contradiction.
\end{proof}  

For ease of notation, let us define $\mathbf{R}=\mathbf{R}_{10}$.

\begin{lemma}\label{sizeR}
    The number of vertices in $\mathbf{R}$ is $o(n)$.
\end{lemma}
	
\begin{proof} By definition, $\mathbf{R}$ is a subset of the set formed by $\bigcup_{i=1}^{10} R_i$ and the vertices in $L$  which we can reach from this union  by paths of length at most $10$ via vertices in $L_8$.
According to Remark~\ref{remark size of R1}, Claim~\ref{claim far} and Lemmas~\ref{lemma few bad pairs with low degrees in G-h},~\ref{no 1-1flat},~\ref{lemma few i vertices that are i-neighbors of h},~\ref{Daras},~\ref{lemma few h33 edges with 1flat and 2flat},~\ref{lemma 3flat 3vertices are adjacent to only 1 1flat 3vertex},~\ref{lemma 1flat 2vtx double neighbor with 5vertex}, and~\ref{lemma 1flat vertices adjacent to 1flat 2 vertices},
$|R_i|\leq 2000 n/\ell$ for each $1\leq i\leq 10$. 

Each vertex of degree at most $8$ has at most $16$ neighbors. Hence the total number of vertices reachable from $\bigcup_{i=1}^{10} R_i$
via paths of length at most $10$ in which all vertices apart from the last ones are in $L_8$ is at most
\[
\left|\bigcup_{i=1}^{10} R_i\right| \cdot \sum_{j=0}^{10} 16^j\leq 20000(n/\ell) 16^{11}=o(n).
\]
\end{proof}  
\begin{lemma}\label{lemma 1flat 4 or 5 is halfhelpful}
    Every  $1$-flat vertex $u\notin \mathbf{R} $ of degree $4$ or $5$ is half-helpful.
\end{lemma}
	
\begin{proof}
	Assume  that  a $1$-flat vertex $u\notin \mathbf{R} $  of degree $4$ or $5$ is not half-helpful. If $u$ is in a $(M,u,2)$ edge with a $1$-flat $2$-vertex, then by Lemma~\ref{lemma 1flat 2 vertex implies rich}, $u$ is not $1$-flat.  So  the only edges  that contain $u$ and a $1$-flat $2$-vertex are low edges. If $u$ is in no such low edges, then $\mathrm{donor}(u)=0$, so $u$ satisfies \eqref{inequality halfhelpful definition} trivially. Thus we may assume $u$ is in at least one edge with a $1$-flat $2$-vertex, and a second vertex that is not helpful. Say $\{u,w,t\}\in E(G)$, where $t$ is a $1$-flat $2$-vertex and $w$ is not helpful. By Lemma~\ref{lemma 1flat vertices adjacent to 1flat 2 vertices}, $u$ and $t$ must both be adjacent to $x$ or both be adjacent to $y$, say $h\in \{x,y\}$ is this common neighbor of $u$ and $t$. 
	
	Let $\{h,t,z\}$ be the second edge of $G$ containing $t$, and note that $hz$ is a rich pair by Lemma~\ref{lemma 1flat 2 vertex implies rich}. Thus, if $z\in \{u,w\}$, this would imply either $hu$ or $hw$ is a rich pair, and thus $\{h,w,t\}$ contains a helpful vertex by Lemma~\ref{lemma rich vertices are helpful}. So, we may assume $z\not\in\{u,w\}$. Let $\{h,u,a\}$ be the non-low edge containing $u$. To avoid a $C_3^{(3)}$ in $G$ with $\{h,t,z\}$ and $\{u,t,w\}$, we must have  $a\in \{z,w\}$. 
	
	\textbf{Case 1:} $a=z$. Then $\{h,u,z\}$ is a rich edge with recipient $u$, so $\mathrm{rich}(u)\geq 1$. If $d(u)=5$, then $u$ satisfies $\eqref{inequality halfhelpful definition}$ even if $\mathrm{donor}(u)=4$, thus  $d(u)=4$. Furthermore, even if $d(u)=4$, $u$ satisfies \eqref{inequality halfhelpful definition} unless $\mathrm{donor}(u)=3$. Let $\{u,t',w'\}$ be a second edge containing $u$, a $1$-flat $2$-vertex $t'$ and a non-helpful vertex $w'$. By Lemma~\ref{lemma 1flat vertices adjacent to 1flat 2 vertices}, $t'$ and $u$ must share a high neighbor, in particular $h$. Let $\{h,t',z'\}\in E(G)$. 
	
	If $z'=w'$, then since $t'\notin \mathbf{R}'_{5} $, $d(w')\geq 4$. Furthermore, as  in this case $d(w't')=2$, by Lemmas \ref{lemma degree 4 double neighbors} and \ref{lemma 1flat 2vtx double neighbor with 5vertex}, we actually have $d(w')\geq 6$. However, then by Lemma~\ref{claim supported or high degree implies helpful} $w'$ is helpful, contradicting our earlier assumption. Thus $z'\neq w'$.
	
	Then the only way the edges $\{h,u,z\}$, $\{h,t',z'\}$ and $\{u,w',t'\}$ do not form a $C_3^{(3)}$ in $G$ is that $z'=z$. Consider the non-edge $\{u,t,t'\}$, and let $T$ be a $C_3^{(3)}$ in $G+\{u,t,t'\}$. Note that $t$ and $t'$ cannot both be core vertices in $T$ since the only edges containing $t$ and $t'$ that do not intersect $\{u,t,t'\}$ in two vertices are $\{h,z,t\}$ and $\{h,z,t'\}$, which do not form a link. Thus, $u$ is a core vertex of $T$, along with one of $t$ and $t'$, say with $t$. In this case, the edge $\{h,t,z\}$ must be one of the edges of $T$, and the second edge must contain $u$ and exactly one vertex in $\{h,z\}$. As $u$ is $1$-flat, this second edge does not contain $h$, so there must be some edge $\{u,z,b\}$. Since $\mathrm{donor}(u)=3$, then we must actually have that $b$ is a $1$-flat $2$-vertex and $z$ is not helpful.  But $hz$ is a rich pair, so we arrive at a contradiction to Lemma~\ref{lemma rich vertices are helpful}.
	
	\textbf{Case 2:} $a=w$. Since $u\notin \mathbf{R}'_{3} $,
	 $d(hw)\geq 2$, say $\{h,w,b\}$ is an edge. To avoid a $C_3^{(3)}$ in $G$, we must have $b=z$. Since $w$ is not helpful, by Lemma~\ref{lemma rich vertices are helpful},  $d(hw)=2$. Furthermore, since $w$ is $2^+$-flat, by Lemmas \ref{lemma 2flat 5vertices are helpful}, \ref{lemma 2flat 4vertices are helpful} and~\ref{claim supported or high degree implies helpful}, $w$ is helpful unless $d(w)\leq 3$. Furthermore, if $d(w)\leq 3$, then by Lemma~\ref{lemma no (L,2,2) edges}, $d(w)=3$. However, in this case, the non-edge $\{h,w,t\}$ intersects every edge containing either $w$ or $t$ in two vertices, so adding $\{h,w,t\}$ to $G$ does not create a $C_3^{(3)}$, a contradiction.
\end{proof}
	
		\section{Lower Bound: Discharging and Final Proof}\label{section discharging}
	We are now ready to present our discharging rules and prove the lower bound. Recall that every vertex in $G$ starts with charge equal to their degree. We then move charge around in $G$ according to the following rules:
	
	\begin{itemize}
	    \item[(D1)] A $(M,M,L)$ edge $\{h,h',u\}$ with $u\in L$ removes charge $1$ from each of $h$ and $h'$ and gives charge $2$ to $u$.
	    \item[(D2)] A needy $(M,L,L)$ edge $\{h,u,t\}$, where $h\in M$ and $t$ is a $2$-vertex or a $1$-flat $3$-vertex removes charge $1$ from $h$ and gives charge $1$ to $t$.
	    \item[(D3)] A rich $(M,L,L)$ edge $\{h,u,v\}$ with $h\in M$ and $v$ the recipient removes charge $1$ from $h$ and gives charge $1$ to $v$.
	    \item[(D4)] A reasonable $(M,L,L)$ edge $\{h,u,v\}$ with $h\in M$ removes charge $1$ from $h$ and gives charge $1/2$ to each of $u$ and $v$.
	    \item[(D5)] All helpful vertices give charge $1/2$ to each of their $1$-flat $2$-neighbors (with multiplicity).
	    \item[(D6)] All vertices that are not helpful but are half-helpful give charge $1/2$ to each of their $1$-flat $2$-neighbors that are not at charge at least $4$ after the application of rules (D1) through (D5).
	\end{itemize}
	
	Now we will use the above discharging scheme to bound the number of edges in $G$.
	
	\begin{theorem}
	The saturation number
	\[
    \mathrm{sat}_3(n,C_3^{(3)})\geq \frac{4n}3-o(n).
    \]
	\end{theorem}
	
	\begin{proof}
	We first prove that the average degree of $G$ is at least $4-o(1)$. Note that since $\ell=\omega(n)$, $|E(G)|\leq 2n$ and $M$ only contains vertices with degree at least $\ell$, $|M|=o(n)$, and consequentially $|L|=n-o(n)$.
	
    We claim that after applying (D1)-(D6), no vertices end up with negative charge. Indeed, first note that vertices in $M$ only lose charge $1$ for each edge they are in, so since the initial charge is their degree, vertices in $M$ end up with non-negative charge. Furthermore, vertices in $L$ that are not helpful or half-helpful do not give away charge, so end up with charge at least equal to their degree. Finally, vertices in $L$ that are helpful or half-helpful end up with at least charge $4$ by the way they are defined. Thus, no vertex ends up with negative charge after applying our discharging scheme.
    
    Now, we claim that all vertices in $L\setminus\mathbf{R}$ end up with charge at least $4$. Let $v\in L\setminus\mathbf{R}$. Note that (D1)-(D6) only removes charge from low vertices if they are helpful or half-helpful, and by the definitions of helpful and half-helpful, these vertices are left with at least charge $4$, so we may assume $v$ is not helpful or half-helpful, and thus $v$ does not give out any charge. Now, let us consider cases based on $d(v)$.
    
    \textbf{Case 1:} $d(v)\leq 1$. Then $v$ is in $R_1$ or in $R'_4$.
    
    \textbf{Case 2:} $d(v)=2$. We need to show that $v$ receives charge at least $2$. If $v$ is supported, (D1) gives $v$ at least charge $2$. If $v$ is not supported but is $2$-flat, then $v$ receives charge $1$ from each non-low edge containing $v$ via (D2). If $v$ is $0$-flat, then $v$ is in $R_1$. Finally, if $v$ is $1$-flat and not in $\mathbf{R}$, then the non-low edge containing $v$ first gives $v$ charge $1$ via (D2), and then by Lemmas \ref{lemma 1flat 2 vertex implies rich} and \ref{lemma rich vertices are helpful}, this non-low edge contains a helpful vertex, and thus it also gives $v$ charge $1/2$ via (D5). The final $1/2$ of charge $v$ needs comes from the low edge containing $v$. Indeed, by Lemma~\ref{lemma low edge containing 1flat 2vtx contains 4+vertex}, this low edge contains a vertex of degree at least $4$, and by Lemma~\ref{claim supported or high degree implies helpful} and Lemmas~\ref{lemma 2flat 4vertices are helpful}, \ref{lemma 2flat 5vertices are helpful} and \ref{lemma 1flat 4 or 5 is halfhelpful}, this vertex is helpful or half-helpful so gives $v$ charge $1/2$ via (D4) or (D5). This gives $v$ total charge at least $4$.
    
    \textbf{Case 3:} $d(v)=3$. Then $v$ needs to get charge at least $1$. If $v$ is $1$-flat, then $v$ gets charge $1$ via its non-low edge by (D2). If $v$ is supported, then it gets charge $2$ via (D1), if $v$ is $2$-flat and not supported, then by Corollary~\ref{remM22}, $v$ is not in a non-low edge with any $2$-vertices, and since $v\notin \mathbf{R}'_6$, $v$ is not in a non-low edge with a $1$-flat $3$-vertex. So both non-low edges containing $v$ are either reasonable or rich with recipient $v$, thus each of these edges gives $v$ charge $1/2$  by (D3) or (D4). Finally if $v$ is $3$-flat and not supported, then by Corollary~\ref{remM22}, $v$ has no $2$-neighbors, and  since $v\notin \mathbf{R}'_7$, $v$ has at most one $1$-flat $3$-neighbor. Thus, $v$ is in at least two non-low edges that do not contain $1$-flat $3$-vertices. These edges will each give charge at least $1/2$ to $v$ as long as they are not rich with a recipient that is not $v$. But since $v$ is not supported and $d(v)<8$, the only way such an edge could be rich with a recipient that is not $v$ is if $hv$ is a rich pair for some $h\in M$. This is not the case since $v\notin \mathbf{R}'_4$. Thus, $v$ gets charge at least $1/2$ via the two non-low edges that contain $v$ and no $1$-flat $3$-vertex by either (D3) or (D4).
    
    \textbf{Case 4:} $d(v)\geq 4$. Since $v\notin R_1$, $v$ is $1^+$-flat, and thus by Lemma~\ref{claim supported or high degree implies helpful} and Lemmas~\ref{lemma 2flat 4vertices are helpful}, \ref{lemma 2flat 5vertices are helpful} and \ref{lemma 1flat 4 or 5 is halfhelpful}, $v$ is either helpful or half-helpful. By the definition of helpful or half-helpful, even after donating some charge via (D5) or (D6), $v$ is left with at least charge $4$.
    
    Thus, every vertex in $L\setminus \mathbf{R}$ ends up with charge at least $4$ and all other vertices end up with non-negative charge. Since the total charge is equal to the total degree, we have that the average degree of $G$ is at least
    \[
    \frac{4|L\setminus\mathbf{R}|}{n}\geq \frac{4((n-o(n))-o(n))}n=4-o(1).
    \]
    Consequently,
    \[
    \mathrm{sat}(n,C_3^{(3)})=|E(G)|\geq \frac{(4-o(1))n}{3}=\frac{4n}3-o(n).
    \]
	\end{proof}


\section{Acknowledgements}
The first author would like to thank Bill Kay and Erin Meger for helpful discussions and ideas involving the topic of $C_3^{(3)}$-saturated graphs.

	\bibliographystyle{plain}
\bibliography{satrefs}
	
\end{document}